\documentclass[11pt,reqno]{amsart}
\usepackage{amsmath,amssymb,amsthm,mathtools}
\usepackage{microtype}
\usepackage{enumitem}
\usepackage{booktabs}
\usepackage{array}
\usepackage{xcolor}
\usepackage{tikz}
\usepackage{float}
\usetikzlibrary{arrows.meta,positioning}
\usepackage[ruled,vlined,norelsize,noend]{algorithm2e}
\numberwithin{equation}{section}
\numberwithin{algocf}{section}
\usepackage[hidelinks]{hyperref}

\usepackage[nameinlink,capitalise,noabbrev]{cleveref}
\SetFuncSty{texttt}
\SetKw{Return}{return}
\crefname{algocf}{algorithm}{algorithms}
\Crefname{algocf}{Algorithm}{Algorithms}

\newtheorem{theorem}{Theorem}[section]
\newtheorem{lemma}[theorem]{Lemma}
\newtheorem{proposition}[theorem]{Proposition}
\newtheorem{corollary}[theorem]{Corollary}
\theoremstyle{definition}
\newtheorem{definition}[theorem]{Definition}
\theoremstyle{remark}
\newtheorem{remark}[theorem]{Remark}

\newcommand{\M}{\mathcal{M}}
\newcommand{\dist}{\operatorname{dist}}
\newcommand{\diam}{\operatorname{diam}}
\newcommand{\gen}{\operatorname{gen}}
\newcommand{\lvl}{\operatorname{lvl}}
\newcommand{\refedge}{e}
\newcommand{\mgen}{\operatorname{mgen}}
\newcommand{\card}{\#}

\newcommand{\BisectTet}{\textnormal{\texttt{BisectTet}}}
\newcommand{\BisectTets}{\textnormal{\texttt{BisectTets}}}
\newcommand{\RefineToConformity}{\textnormal{\texttt{RefineToConformity}}}
\SetKwBlock{ConformityBlock}{\RefineToConformity:}{}
\newcommand{\BanschBisectTet}{\textnormal{\texttt{B\"anschBisectTet}}}
\newcommand{\BanschBisectTets}{\textnormal{\texttt{B\"anschBisectTets}}}
\newcommand{\BanschRefine}{\textnormal{\texttt{B\"anschRefine}}}
\makeatletter
\newenvironment{fixedfigure}
  {\par\medskip\noindent\begin{minipage}{\linewidth}\def\@captype{figure}\centering}
  {\end{minipage}\par\medskip}
\makeatother

\title[Closure complexity for AMP and B\"ansch]
{Closure complexity of
B\"ansch-type algorithms for tetrahedral mesh refinement}
\author[Y. Li]{Yuwen Li}
\address{School of Mathematical Sciences, Zhejiang University,
866 Yuhangtang Road, Hangzhou, Zhejiang 310058, People's Republic of China}
\email{liyuwen@zju.edu.cn}
\author[Z. Yang]{Zhiyuan Yang}
\address{School of Mathematical Sciences, Zhejiang University,
866 Yuhangtang Road, Hangzhou, Zhejiang 310058, People's Republic of China}
\email{zhiyuanyang@zju.edu.cn}
\date{August 2026}

\begin{document}

\begin{abstract}
We prove, to our knowledge, the first unconditional cumulative closure
estimates for the Arnold--Mukherjee--Pouly (AMP) refinement algorithm and the
original face-marked tetrahedral algorithm of B\"ansch on arbitrary conforming
initial tetrahedral meshes.  Let
$\mathcal T_0,\ldots,\mathcal T_L$ be an adaptive mesh sequence generated by
either algorithm, with $\M_\ell\subseteq\mathcal T_\ell$ denoting the marking
set at step $\ell$.  Then
\[
  \card\mathcal{T}_L-\card\mathcal{T}_0
  \le C_{\mathrm{clos}}(\mathcal{T}_0)
      \sum_{\ell=0}^{L-1}\card\M_\ell.
\]
The proof is intrinsic to the physical three-dimensional mesh and requires neither an
initial compatibility condition nor a higher-dimensional embedding.  It
organizes conformity refinements into a causal forest and combines a uniform
horizontal-propagation estimate with a weighted packing argument to obtain an
explicit closure constant.  For the original B\"ansch algorithm, every
history-dependent resolution of the initial two-edge ambiguity is represented
by one of finitely many AMP histories.  The estimate therefore holds uniformly
for arbitrary deterministic or nondeterministic choices.  This resolves a
long-standing complexity question for the B\"ansch--AMP family.
\end{abstract}

\maketitle

\section{Introduction}\label{sec:introduction}

Adaptive finite element methods (AFEMs) rely on local mesh refinement to
concentrate computational effort near singularities, layers, interfaces, and
other fine-scale features.  A standard adaptive feedback loop is of the form
\[
\texttt{Solve}\rightarrow\texttt{Estimate}\rightarrow\texttt{Mark}\rightarrow\texttt{Refine}.
\]
Convergence and optimality of AFEMs have been established in
\cite{Dorfler1996,MNS2002,BDD2004,Stevenson2007,CasconEtAl2008,CFPP2014}.
Besides conformity and uniform shape regularity, optimal complexity of AFEMs
requires the \texttt{Refine} routine to satisfy a quantitative complexity
property.  In particular, conforming completion may create more elements than
were explicitly marked, but its cumulative overhead must remain proportional
to the cumulative number of marked elements.  Let $\M_\ell$ be the marking set
at the $\ell$th AFEM iteration, namely the set of elements selected for
refinement, and
\[
 \mathcal{T}_{\ell+1}=\texttt{Refine}(\mathcal{T}_\ell,\M_\ell),
 \qquad \ell=0,\ldots,L-1
\]
be a sequence of locally refined meshes output by an AFEM\@.
The desired closure estimate has the form
\begin{equation}\label{eq:introclosure}
 \card\mathcal{T}_L-\card\mathcal{T}_0
 \le C_{\mathrm{clos}}(\mathcal{T}_0)
      \sum_{\ell=0}^{L-1}\card\M_\ell.
\end{equation}
This closure estimate is one of the foundations of rate-optimal AFEM
theory~\cite{BDD2004,Stevenson2008,BonitoEtAl2024}.  In three dimensions, its
proof is particularly delicate because one bisection may launch a long chain
of conformity refinements through faces and edge stars, and repeated calls may
revisit the same physical region at many scales.

In modern AFEMs, bisection has become a standard local-refinement mechanism
because it preserves nestedness and reduces mesh management to a small number
of local operations.  For triangular meshes, longest-edge
bisection~\cite{RosenbergStenger1975,Rivara1984,RivaraVenere1996,SuarezPlazaCarey2005}
and newest-vertex bisection~\cite{Sewell1972,Mitchell1989,Mitchell2016} are the
predominant bisection algorithms.  The passage to tetrahedral and simplicial
meshes produced several distinct refinement paradigms.  B\"ansch proposed a
face-marked bisection method with an ambiguity in two and three
dimensions~\cite{Bansch1991} that works on arbitrary initial meshes. Liu and Joe proposed a related bisection algorithm with uniform shape
regularity and local grading \cite{LiuJoe1995}.
Arnold--Mukherjee--Pouly (AMP) later recast the B\"ansch rule in a particularly
transparent marked-tetrahedron data structure and proved conformity,
termination, and finitely many similarity classes~\cite{AMP2000}.

In a different paradigm, Kossaczk\'y
developed a recursive newest-vertex-type construction in two and three
dimensions under restrictions on the initial mesh~\cite{Kossaczky1994}.
Maubach and Traxler extended this construction to arbitrary dimension by
assigning an ordered vertex list to every simplex, which hinges on compatibility across neighboring simplices and imposes
a structural condition on the coarse mesh~\cite{Maubach1995,Traxler1997}.
By contrast, B\"ansch--AMP is face-based: the input data are edge marks
intrinsic to the triangular faces, and in three dimensions these marks can be
initialized directly on any conforming tetrahedral mesh from a single global
edge order, without a compatible simplex-ordering construction.

Closure-complexity estimates for bisection algorithms developed along a
somewhat different line.  Binev, Dahmen, and DeVore~\cite{BDD2004} proved the
decisive cumulative estimate for two-dimensional newest-vertex bisection under
an initial compatibility condition.  Stevenson extended such estimates to
arbitrary dimension under a matching-neighbor
condition~\cite{Stevenson2008}.  In two dimensions, Karkulik, Pavlicek, and
Praetorius proved that newest-vertex bisection satisfies
\eqref{eq:introclosure} on every conforming initial triangular mesh, without
any compatibility assumption~\cite{KPP2013,KPP2015Erratum}.

In dimensions higher than two, the usual matching-neighbor assumptions may
not hold on general triangulations.  Alk\"amper--Gaspoz--Kl\"ofkorn (AGK)
introduced a weak compatibility condition and a linear-time relabeling
procedure that ensures termination of iterative refinement in arbitrary
dimension~\cite{AGK2018}.  Gehring subsequently proved
\eqref{eq:introclosure} for every initial labeling produced by the AGK
initialization algorithm~\cite{Gehring2023}.  Most recently,
Diening--Gehring--Storn (DGS) supplied the Maubach routine with an
initialization valid for every conforming initial triangulation and proved
closure by relating the physical mesh to a suitably colored triangulation in
a higher-dimensional space~\cite{DGS}.

Existing closure-complexity analyses follow the line of Maubach--Traxler.
The B\"ansch--AMP family occupies a distinctive place in this history.  The
B\"ansch algorithm, introduced in 1991, acts directly on a given conforming tetrahedral mesh,
without a coloring or matching-neighbor assumption~\cite{Bansch1991}. The AMP algorithm is its
later finite-state realization.  Indeed, every initial mesh can be
initialized by marking each face with its maximal edge in one global edge
order.  Thus no preliminary refinement or global compatibility construction
is required.  This unusually simple initialization is one of the principal
advantages of B\"ansch--AMP\@.

Despite that early generality, the B\"ansch--AMP closure complexity of the form
\eqref{eq:introclosure} has remained an open question since the original
works~\cite{Bansch1991,AMP2000}.  The present manuscript resolves this
long-standing problem in modern AFEM optimality theory.  To our
knowledge, it gives the first unconditional cumulative closure estimates for
B\"ansch and AMP on arbitrary conforming initial tetrahedral meshes, including
the nondeterministic choice allowed by the original B\"ansch formulation.  The
proof is intrinsic to the physical three-dimensional mesh: it uses neither a
higher-dimensional embedding nor an initial compatibility hypothesis.  The
proof requires a new causal analysis of conformity propagation rather than a direct
application of classical newest-vertex-bisection completion theory.

The main contributions are the following.
\begin{enumerate}[label=\textnormal{(\roman*)},leftmargin=2.4em]
\item We construct intrinsic marked-face refinement trees and a causal forest
for conformity events.  The construction gives the macro-generation
monotonicity needed to control propagation.  A finite analysis of the five
physical edge classes, two transition tables, and vertex- and edge-star bounds
then proves uniform horizontal propagation within a fixed macro-generation.
Writing $\nu=\nu(\mathcal T_0)$ for the resulting uniform macro-vertex valence,
the bound is
\[
 H_{\mathrm{AMP}}
 =\max\{448\,\card\mathcal{T}_0,630\,\nu\}.
\]

\item We convert horizontal boundedness into geometric locality by showing
that every causal descendant remains in a controlled same-scale neighborhood
of its root event.  Let $c_V$ and $C_h$ denote, respectively, the lower
volume-scale and upper diameter-scale constants from \Cref{lem:scale}.  An
optimized self-contained weighted packing argument then charges the resulting
refinement cost to the marked elements and yields
\[
 C_{\mathrm{clos}}(\mathcal{T}_0)
 \le36\,\frac{\pi C_h^3}{c_V}
 (4H_{\mathrm{AMP}}+1)(4H_{\mathrm{AMP}}+3)^2.
\]

\item We resolve the two-edge choice in the original B\"ansch formulation.
Every B\"ansch child has a unique global refinement edge, and a
deferred-resolution argument represents every history-dependent B\"ansch
history by one of finitely many AMP histories.  The AMP closure estimate
therefore transfers uniformly to arbitrary deterministic or nondeterministic
choices.
\end{enumerate}

The result is three-dimensional: both the AMP marking rules and the finite
transition analysis are specific to tetrahedra.  \Cref{sec:amp-algorithm}
restates the algorithm and its three-bisection macrostructure.
\Cref{sec:amp-analysis} develops the intrinsic causal structure.
\Cref{sec:horizontal-propagation} proves the horizontal-propagation bound by
a finite macrostep analysis, and \Cref{sec:amp-closure} derives locality and
closure from that bound.
Finally, \Cref{sec:bansch-analysis} states the B\"ansch rules intrinsically and
derives the uniform closure consequence.
\Cref{sec:numerics} reports the numerical experiments.

\section{The AMP refinement algorithm}
\label{sec:amp-algorithm}

This section fixes the AMP data structure and refinement rules used throughout
the paper.  The material is restated from
Arnold--Mukherjee--Pouly~\cite{AMP2000}, with notation adapted
to the present analysis.  Their marked-tetrahedron formulation is essentially
equivalent to the earlier face-marked algorithm of B\"ansch once one
refinement edge is selected in the ambiguous two-edge case.  We return to the original B\"ansch formulation in
\Cref{sec:bansch-analysis}.

\subsection{Tetrahedral meshes and marked tetrahedra}

Let $\mathcal{T}_0$ be an arbitrary conforming tetrahedral triangulation of a polyhedral
domain $\Omega\subset\mathbb R^3$.  A tetrahedral mesh $\mathcal{T}$ is a finite
collection of closed tetrahedra such that distinct tetrahedra have disjoint
interiors and $\bigcup_{K\in\mathcal{T}}K=\overline\Omega$.  For a tetrahedron or a
mesh $X$, write $\mathcal{V}(X)$, $\mathcal{E}(X)$, and $\mathcal{F}(X)$ for its sets of vertices,
edges, and faces.  The mesh is \emph{conforming} if the intersection of two
distinct tetrahedra is empty or is a common vertex, a common edge, or a common
face.

\begin{definition}[Marked tetrahedron]\label{def:markedtetrahedron}
A marked tetrahedron $K$ consists of its four vertices, a refinement edge
$\refedge(K)\in\mathcal{E}(K)$, a marked edge $\mu_F\in\mathcal{E}(F)$ for every
$F\in\mathcal{F}(K)$, and a Boolean flag.  The two faces containing
$\refedge(K)$ are the \emph{refinement faces}, and on each of them
$\mu_F=\refedge(K)$.  The remaining two faces are the
\emph{nonrefinement faces}.  The flag may be set only when the four face marks
are coplanar.
\end{definition}

The marked edges of the two nonrefinement faces are each either adjacent or
opposite to the refinement edge.  This gives the AMP types:
\begin{itemize}[leftmargin=2.1em]
\item type $P$ (planar), if all face marks are coplanar; the type is denoted by
      $P_u$ or $P_f$ according as the flag is unset or set;
\item type $A$ (adjacent), if the marked edges of both nonrefinement faces meet the refinement edge
      but the four marks are not coplanar;
\item type $O$ (opposite), if neither nonrefinement face has its marked edge
      meeting the refinement edge;
\item type $M$ (mixed), if the marked edge of exactly one nonrefinement face meets the refinement
      edge.
\end{itemize}
Thus $P_f$ is the only type whose flag is set; the flags of $P_u,A,O$, and $M$
are unset.
\Cref{fig:AMPtypes} displays one representative of each type both as a marked
tetrahedron and as an unfolded face net.

\begin{figure}[!ht]
\centering
\begin{tikzpicture}[
  faceedge/.style={line width=.38pt},
  facemark/.style={line width=1.15pt},
  tetedge/.style={line width=.45pt,line join=round},
  tethidden/.style={densely dashed,line width=.40pt},
  tetmark/.style={line width=1.15pt,line cap=round},
  dualmark/.style={double=white,double distance=.75pt,line width=.42pt,
                   line cap=round},
  every node/.style={font=\footnotesize}
]
\newcommand{\AMPtetra}{%
  \coordinate (ta) at (-.72,-.22);
  \coordinate (tb) at ( .00,-.60);
  \coordinate (tc) at ( .04, .80);
  \coordinate (td) at ( .72,-.18);
  \draw[tethidden] (ta)--(td);
  \draw[tetedge] (ta)--(tb)--(td);
  \draw[tetedge] (tc)--(ta) (tc)--(tb) (tc)--(td);
}
\newcommand{\AMPnet}{%
  \draw[faceedge] (-.55,0)--(.55,0)--(0,-.95)--cycle;
  \draw[faceedge] (-.55,0)--(0,.95)--(.55,0);
  \draw[faceedge] (-.55,0)--(-1.10,-.95)--(0,-.95);
  \draw[faceedge] (.55,0)--(1.10,-.95)--(0,-.95);
}
\newcommand{\AMPr}{%
  \draw[facemark] (-.47,.07)--(.47,.07);
  \draw[facemark] (-.47,-.07)--(.47,-.07);
}
\newcommand{\AMPac}{\draw[facemark] (-.57,-.10)--(-.10,-.90);}
\newcommand{\AMPbc}{\draw[facemark] (.57,-.10)--(.10,-.90);}
\newcommand{\AMPbd}{\draw[facemark] (.53,-.10)--(1.00,-.90);}
\newcommand{\AMPcdleft}{\draw[facemark] (-1.00,-.88)--(-.10,-.88);}
\newcommand{\AMPcdright}{\draw[facemark] (.10,-.88)--(1.00,-.88);}

\begin{scope}[xshift=0cm]
  \begin{scope}[yshift=1.15cm,scale=1.15]
    \AMPtetra
    \draw[dualmark] (ta)--(tb);
    \draw[tetmark] (ta)--(tc) (tb)--(tc);
  \end{scope}
  \begin{scope}[yshift=-.65cm]
    \AMPnet\AMPr\AMPac\AMPbc
  \end{scope}
  \node at (0,-1.78) {$P$ ($P_u$ or $P_f$)};
\end{scope}
\begin{scope}[xshift=3.0cm]
  \begin{scope}[yshift=1.15cm,scale=1.15]
    \AMPtetra
    \draw[dualmark] (ta)--(tb);
    \draw[tetmark] (ta)--(tc) (tb)--(td);
  \end{scope}
  \begin{scope}[yshift=-.65cm]
    \AMPnet\AMPr\AMPac\AMPbd
  \end{scope}
  \node at (0,-1.78) {$A$};
\end{scope}
\begin{scope}[xshift=6.0cm]
  \begin{scope}[yshift=1.15cm,scale=1.15]
    \AMPtetra
    \draw[dualmark] (ta)--(tb);
    \draw[dualmark] (tc)--(td);
  \end{scope}
  \begin{scope}[yshift=-.65cm]
    \AMPnet\AMPr\AMPcdleft\AMPcdright
  \end{scope}
  \node at (0,-1.78) {$O$};
\end{scope}
\begin{scope}[xshift=9.0cm]
  \begin{scope}[yshift=1.15cm,scale=1.15]
    \AMPtetra
    \draw[dualmark] (ta)--(tb);
    \draw[tetmark] (ta)--(tc);
    \draw[tetmark] (tc)--(td);
  \end{scope}
  \begin{scope}[yshift=-.65cm]
    \AMPnet\AMPr\AMPac\AMPcdright
  \end{scope}
  \node at (0,-1.78) {$M$};
\end{scope}
\end{tikzpicture}
\caption{The four AMP types, shown as marked tetrahedra (top) and unfolded face
nets (bottom).  Thick strokes are face marks, with double strokes for an edge
marked on both incident faces.  The states $P_u$
and $P_f$ differ only in the Boolean flag.}
\label{fig:AMPtypes}
\end{figure}
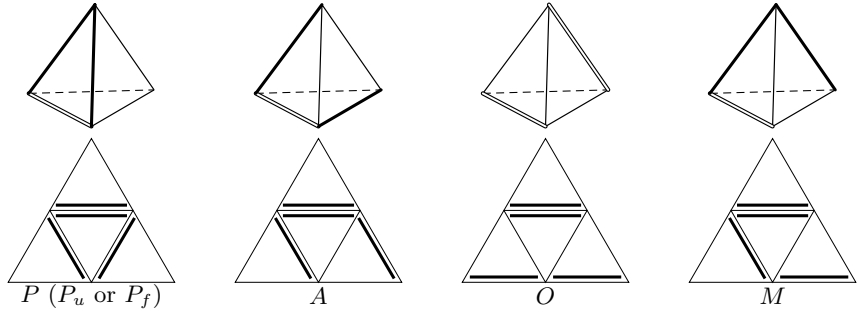

An \emph{AMP marked mesh} is a tetrahedral mesh whose tetrahedra carry the
marked-tetrahedron structure of \Cref{def:markedtetrahedron}.  We use the same
symbol for the marked mesh and its underlying tetrahedral collection whenever
only membership or cardinality is concerned.  It is \emph{conformingly marked}
if its underlying mesh is conforming and the markings induced by two incident
tetrahedra on any shared face $F$ agree; that is, both assign the same edge
$\mu_F$ to $F$.

\subsection{The marked-tetrahedron bisection}

Fix $K=[v_0,v_1,v_2,v_3]$ with refinement edge $[v_0,v_1]$, and write
$m:=\operatorname{mid}([v_0,v_1])$.  In the rules below, an inherited face is
a face of $K$, a cut face is a proper subtriangle of a face of $K$, and the
new face common to the children is $[m,v_2,v_3]$.
\begin{enumerate}[label=\textnormal{(B\arabic*)},leftmargin=2.8em]
\item The inherited face keeps its marked edge from $K$; this edge is the
      refinement edge of the child.
\item A cut face is marked by its edge opposite the new vertex $m$.
\item The common new face is marked by the edge opposite the new vertex $m$,
      except when $K$ is of type $P_f$; then it is marked by $[m,w]$, where
      $w\in\{v_2,v_3\}$ is the common old endpoint of the two child
      refinement edges.
\item The child flags are set if and only if $K$ is of type $P_u$.
\end{enumerate}

\begin{algorithm}[H]
\caption{One AMP marked-tetrahedron bisection \BisectTet}
\label{alg:bisecttet}
\KwIn{A marked tetrahedron $K$ with refinement edge $[v_0,v_1]$}
\KwOut{the two marked children $\{K_0, K_1\}=\BisectTet(K)$}
$m\gets\operatorname{mid}([v_0,v_1])$\;
$K_0\gets[v_0,m,v_2,v_3]$ and
$K_1\gets[m,v_1,v_2,v_3]$\;
Assign marks and flags according to
\textnormal{(B1)}--\textnormal{(B4)}\;
\Return $\{K_0,K_1\}$\;
\end{algorithm}

\begin{fixedfigure}
\begin{tikzpicture}[scale=1.0,
  every node/.style={font=\small},
  oldedge/.style={line width=1.2pt},
  newedge/.style={densely dashed,line width=0.9pt}]
\coordinate (v0) at (-2.2,-0.8);
\coordinate (v1) at ( 2.2,-0.8);
\coordinate (v2) at ( 0.8, 1.5);
\coordinate (v3) at (-0.1, 0.15);
\coordinate (m)  at ( 0.0,-0.8);
\draw (v0)--(v2)--(v1)--cycle;
\draw (v0)--(v3)--(v1);
\draw[densely dashed] (v3)--(v2);
\draw[oldedge] (v0)--(v1);
\fill (v0) circle (1.3pt) node[below left] {$v_0$};
\fill (v1) circle (1.3pt) node[below right] {$v_1$};
\fill (v2) circle (1.3pt) node[above] {$v_2$};
\fill (v3) circle (1.3pt) node[above left] {$v_3$};
\fill (m) circle (1.7pt) node[below] {$m$};
\draw[newedge] (m)--(v2) (m)--(v3);
\node[align=center] at (0,-1.65)
 {$K_0=[v_0,m,v_2,v_3]$\qquad$K_1=[m,v_1,v_2,v_3]$};
\end{tikzpicture}
\caption{Bisection of a tetrahedron along its refinement edge.  The dashed
segments from $m$ are the new edges; the two children share the new face
$[m,v_2,v_3]$.}
\label{fig:bisecttet}
\end{fixedfigure}

For an explicit instance of the four rules, take a type-$P_u$ tetrahedron
with unset flag and face marks
\[
 \begin{aligned}
 \mu_{[v_0,v_1,v_2]}&=\mu_{[v_0,v_1,v_3]}=[v_0,v_1],\\
 \mu_{[v_0,v_2,v_3]}&=[v_0,v_2],\qquad \mu_{[v_1,v_2,v_3]}=[v_1,v_2].
 \end{aligned}
\]
The child data are listed in \Cref{tab:Pu-bisection}.
\begin{table}[!ht]
\centering
\small
\renewcommand{\arraystretch}{1.15}
\begin{tabular}{@{}c@{\qquad}cc@{\qquad}cc@{}}
\toprule
& \multicolumn{2}{c}{$K_0=[v_0,m,v_2,v_3]$}
& \multicolumn{2}{c}{$K_1=[m,v_1,v_2,v_3]$}\\
\cmidrule(lr){2-3}\cmidrule(l){4-5}
face kind & face & mark & face & mark\\
\midrule
inherited & $[v_0,v_2,v_3]$ & $[v_0,v_2]$
          & $[v_1,v_2,v_3]$ & $[v_1,v_2]$\\
cut       & $[v_0,m,v_2]$   & $[v_0,v_2]$
          & $[m,v_1,v_2]$   & $[v_1,v_2]$\\
cut       & $[v_0,m,v_3]$   & $[v_0,v_3]$
          & $[m,v_1,v_3]$   & $[v_1,v_3]$\\
new       & $[m,v_2,v_3]$   & $[v_2,v_3]$
          & $[m,v_2,v_3]$   & $[v_2,v_3]$\\
flag/type & \multicolumn{2}{c}{set; $P_f$}
          & \multicolumn{2}{c}{set; $P_f$}\\
\bottomrule
\end{tabular}
\caption{The rules \textnormal{(B1)}--\textnormal{(B4)} for a representative
$P_u\to P_f$ bisection.}
\label{tab:Pu-bisection}
\end{table}

Both children have the same type.  The complete type transition is shown in
\Cref{fig:AMPtypetransition}.  In particular, types $M$ and $O$ occur only in
the initial transient.

\begin{fixedfigure}
\begin{tikzpicture}[
  every node/.style={font=\small},
  flow/.style={->,line width=.45pt}
]
\node (A)  at (0,1.15) {$A$};
\node (Pf) at (2.60,1.15) {$P_f$};
\node (Pu) at (1.30,0) {$P_u$};
\node (M)  at (0,-1.15) {$M$};
\node (O)  at (2.60,-1.15) {$O$};
\draw[flow] (A)--(Pu);
\draw[flow] (Pu)--(Pf);
\draw[flow] (Pf)--(A);
\draw[flow] (M)--(Pu);
\draw[flow] (O)--(Pu);
\end{tikzpicture}
\caption{The AMP type transition under one bisection.  Each arrow points from
the parent type to the common type of its two children.}
\label{fig:AMPtypetransition}
\end{fixedfigure}

Repeated bisection organizes the descendants of each initial tetrahedron into
a binary refinement tree.  The \emph{generation} $\gen(K)$ of a descendant
$K$ is its depth in this tree, namely the number of bisections from its unique
ancestor in $\mathcal{T}_0$.  Thus every bisection increases generation by one and
halves volume.  At an intermediate stage, an \emph{active tetrahedron} is a
leaf of one of these trees: it has been created and has not itself been
bisected.  The active tetrahedra form the current, possibly nonconforming,
mesh $\mathcal{T}$.  A current mesh vertex $z\in\mathcal{V}(\mathcal{T})$ is a
\emph{hanging vertex} of an active tetrahedron $K\in\mathcal{T}$ if
\[
 z\in\partial K\setminus\mathcal{V}(K).
\]
Equivalently, $z$ lies in the relative interior of an edge or a face of $K$;
an arbitrary point there is not a hanging vertex.

\begin{fixedfigure}
\begin{tikzpicture}[
  every node/.style={font=\footnotesize},
  tetedge/.style={line width=.45pt,line join=round},
  hidden/.style={densely dashed,black!55,line width=.45pt,line join=round},
  refedge/.style={line width=1.05pt,line cap=round},
  newedge/.style={line width=.85pt,line cap=round}
]
\begin{scope}[xshift=-2.85cm,yshift=.18cm]
  \coordinate (a) at (-1.55,0);
  \coordinate (b) at ( 1.55,0);
  \coordinate (c) at ( .35,.75);
  \coordinate (p) at (-.25,1.65);
  \coordinate (q) at (-.35,-.95);

  \fill[black!4] (a)--(b)--(c)--cycle;
  \draw[hidden] (q)--(a) (q)--(b) (q)--(c);
  \draw[tetedge] (a)--(c)--(b);
  \draw[tetedge] (p)--(a) (p)--(b) (p)--(c);
  \draw[refedge] (a)--(b);

  \fill (a) circle (1.15pt) node[below left] {$a$};
  \fill (b) circle (1.15pt) node[below right] {$b$};
  \fill (c) circle (1.15pt) node[right] {$c$};
  \fill (p) circle (1.15pt) node[above] {$p$};
  \fill (q) circle (1.15pt) node[below] {$q$};

  \node[anchor=east] at (-.82,1.28) {$K^+$};
  \node[anchor=west] at (.18,-.78) {$K^-$};
  \node[align=center] at (0,-1.60)
    {\textnormal{(a)} before bisection};
\end{scope}

\begin{scope}[xshift=2.85cm,yshift=.18cm]
  \coordinate (a) at (-1.55,0);
  \coordinate (b) at ( 1.55,0);
  \coordinate (c) at ( .35,.75);
  \coordinate (p) at (-.25,1.65);
  \coordinate (q) at (-.35,-.95);
  \coordinate (m) at (0,0);

  \fill[black!3] (a)--(m)--(c)--cycle;
  \fill[black!7] (m)--(b)--(c)--cycle;
  \draw[hidden] (q)--(a) (q)--(b) (q)--(c);
  \draw[tetedge] (a)--(c)--(b);
  \draw[tetedge] (p)--(a) (p)--(b) (p)--(c);
  \draw[tetedge] (a)--(m)--(b);
  \draw[newedge] (m)--(c) (m)--(p);

  \fill (a) circle (1.15pt) node[below left] {$a$};
  \fill (b) circle (1.15pt) node[below right] {$b$};
  \fill (c) circle (1.15pt) node[right] {$c$};
  \fill (p) circle (1.15pt) node[above] {$p$};
  \fill (q) circle (1.15pt) node[below] {$q$};
  \fill (m) circle (1.55pt) node[below right] {$m$};

  \node[align=right,anchor=east] at (-.58,1.38)
    {children\\
     of $K^+$};
  \node[anchor=west] at (.18,-.78) {$K^-$};
  \node[align=center] at (0,-1.60)
    {\textnormal{(b)} after bisection; $m$ is hanging for $K^-$};
\end{scope}
\end{tikzpicture}

\caption{A hanging vertex created on a tetrahedral interface.  Initially,
$K^+=[a,b,c,p]$ and $K^-=[a,b,c,q]$ share the face $[a,b,c]$, and $[a,b]$
is the refinement edge of $K^+$.  After $K^+$ is bisected, its midpoint $m$
is a vertex of the two children but not of $K^-$.  Dashed segments indicate
edges of $K^-$ lying behind the shared face.}
\label{fig:hangingvertex}
\end{fixedfigure}

In \Cref{fig:hangingvertex}, $K^-$ and the two children of $K^+$ are all
active, but only $K^-$ has $m$ as a hanging vertex.  Thus the active
tetrahedra with hanging vertices form the subset selected in the next
conformity batch.

\subsection{Conforming initialization and local refinement}

An \emph{initial AMP marked mesh} is a finite conformingly marked AMP mesh whose flags
are all unset.  A convenient canonical initialization imposes one strict
global order on all edges of $\mathcal{T}_0$.  The maximal edge of each tetrahedron is
its refinement edge, the maximal edge of each face is its marked edge, and
every flag is unset.  A shared face therefore receives the same mark from both
sides.  No coloring or matching-neighbor condition is imposed on $\mathcal{T}_0$.
Henceforth, dependence on an initial AMP marked mesh $\mathcal{T}_0$ includes its fixed
refinement edges, face marks, and flags; the global edge order is only one way
to construct these data.

Let $\mathcal{T}$ be a current AMP marked mesh and let $\mathcal S\subseteq\mathcal{T}$.  We write
\BisectTets{} for the routine that replaces every tetrahedron of $\mathcal S$
by the two children returned by \Cref{alg:bisecttet}; all other tetrahedra are
unchanged.

\begin{algorithm}[H]
\caption{Simultaneous AMP bisection \BisectTets}
\label{alg:bisecttets}
\KwIn{A current AMP marked mesh $\mathcal{T}$ and a set $\mathcal S\subseteq\mathcal{T}$}
\KwOut{the AMP marked mesh $\BisectTets(\mathcal{T},\mathcal S)$ obtained by bisecting
all tetrahedra in $\mathcal S$}

$\widehat{\mathcal{T}}\gets\mathcal{T}\setminus\mathcal S$\;

\For{$K\in\mathcal S$}{
  $\{K_1, K_2\}\gets\BisectTet(K)$\;
  $\widehat{\mathcal{T}}\gets\widehat{\mathcal{T}}\cup\{K_1,K_2\}$\;
}

\Return $\widehat{\mathcal{T}}$\;
\end{algorithm}

The AMP local-refinement procedure~\cite{AMP2000} is summarized in
\Cref{alg:localrefine}.

\begin{algorithm}[H]
\caption{One AMP refinement call \texttt{RefineAMP}}
\label{alg:localrefine}
\KwIn{A conformingly marked AMP mesh $\mathcal{T}$ and a marking set $\M\subseteq\mathcal{T}$}
\KwOut{the AMP marked mesh $\mathcal{T}^+=\texttt{RefineAMP}(\mathcal{T},\M)$ generated
from $\M$}
$\mathcal{T}^+\gets\BisectTets(\mathcal{T},\M)$\;

\ConformityBlock{
  $\mathcal S\gets\{K\in\mathcal{T}^+:K\text{ has a hanging vertex}\}$\;
  \While{$\mathcal S\ne\varnothing$}{
    $\mathcal{T}^+\gets\BisectTets(\mathcal{T}^+,\mathcal S)$\;
    $\mathcal S\gets\{K\in\mathcal{T}^+:K\text{ has a hanging vertex}\}$\;
  }
}
\Return $\mathcal{T}^+$\;
\end{algorithm}

Thus the input--output map of \Cref{alg:localrefine} is
\begin{equation*}
 \texttt{RefineAMP}(\mathcal{T},\M)
 :=\RefineToConformity\bigl(\BisectTets(\mathcal{T},\M)\bigr).
\end{equation*}
Despite its name, the \RefineToConformity{} module need not return a conforming mesh for
an arbitrary marked input: in three dimensions, a mesh without hanging
vertices may still be nonconforming.  However, \Cref{thm:AMPtermination}
shows that, along a refinement history issued from an initial AMP marked mesh,
the module terminates and returns a conformingly marked AMP mesh.

\begin{theorem}[AMP termination and conformity]\label{thm:AMPtermination}
Let $\mathcal{T}_0$ be an initial AMP marked mesh.  Every \texttt{\emph{RefineAMP}} call along the resulting
refinement history terminates and produces a conformingly marked AMP mesh.
Moreover, if
\[
 \mathcal{T}_{\ell+1}=\texttt{\emph{RefineAMP}}(\mathcal{T}_\ell,\M_\ell),
 \qquad \ell=0,1,\ldots,
\]
then every tetrahedron of $\mathcal{T}_\ell$ has generation at most $3\ell$.
\end{theorem}

\begin{proof}
This is the termination and generation theorem of
Arnold--Mukherjee--Pouly~\cite[Theorem~3.1]{AMP2000}.
\end{proof}

\subsection{The three-bisection macrostructure}

Let $\mathcal{U}_q$ be the mesh obtained by applying $\BisectTet$ uniformly $3q$
times to every initial tetrahedron.  We call $\mathcal{U}_q$ the $q$th
\emph{macro-mesh}.  The AMP three-bisection structure partitions the
geometric edges occurring in the uniform refinement tree into the sets
$\mathcal{E}(\mathcal{U}_q)$, $q\ge0$, see
\cite{AMP2000}.  We call $q$ the
\emph{macro-generation} of an edge $e\in\mathcal{E}(\mathcal{U}_q)$ and write
\[
 \mgen(e)=q
 \quad\Longleftrightarrow\quad
 e\in\mathcal{E}(\mathcal{U}_q).
\]
We call the passage from $\mathcal{U}_q$ to $\mathcal{U}_{q+1}$ one AMP
\emph{macrostep}.
For a tetrahedron $K$, set
\begin{equation*}
 \lvl(K):=\left\lfloor\frac{\gen(K)}3\right\rfloor.
\end{equation*}

\begin{proposition}[AMP macrostructure]\label{prop:AMPmacrostructure}
The uniform AMP refinement tree has the following properties.
\begin{enumerate}[label=(A\arabic*),leftmargin=2.8em]
\item \label{A1}
Every $\mathcal{U}_q$ is conformingly marked and unflagged.  For $q\ge1$, each
tetrahedron of $\mathcal{U}_q$ is of type $A$ or $P_u$.
\item \label{A2}
If $\gen(K)=3q+s$, $s\in\{0,1,2\}$, then the possible types of $K$ and
the macro-generations of its edges are
\[
\begin{array}{c|c|c}
s&\text{types of }K&\text{macro-generations of the edges of }K\\ \hline
0&A,\ P_u\text{; also }M,O\text{ if }q=0&q,\\
1&P_u,\ P_f&q\text{ or }q+1,\\
2&P_f,\ A&q\text{ or }q+1.
\end{array}
\]
\item \label{A3}
The descendants of the fixed finite initial mesh belong to finitely many
similarity classes and are uniformly shape regular.
\end{enumerate}
\end{proposition}

\begin{proof}
Items~\ref{A1}--\ref{A2} are proved
in~\cite[Lemmas~3.2--3.3 and Proposition~3.4]{AMP2000}.  Item~\ref{A3} follows
from the finite-similarity-class theorem; see
\cite[Proposition~4.1 and Theorems~4.2 and~4.5]{AMP2000}.
\end{proof}

\section{Intrinsic causal structure}
\label{sec:amp-analysis}

The algorithm and macrostructure recalled in the preceding section are due to
Arnold--Mukherjee--Pouly.  We now derive the intrinsic refinement trees and
causal chains associated with one refinement call and prove their
macro-generation monotonicity, together with the macro-star counting
consequences used below.  The resulting horizontal-propagation bound is proved
in \Cref{sec:horizontal-propagation}; its consequences for geometric locality
and closure complexity are developed in \Cref{sec:amp-closure}.

\subsection{Geometry of AMP descendants}

By the shape-regularity statement in \Cref{prop:AMPmacrostructure}, the
descendants of the fixed finite initial mesh are uniformly shape regular, with
constants depending only on $\mathcal{T}_0$ and its initial markings.

\begin{lemma}[Scale bounds]\label{lem:scale}
There are constants $c_V,C_V,c_h,C_h>0$, depending only on the fixed initial
mesh $\mathcal{T}_0$, such that every AMP descendant $K$ satisfies
\begin{align}
 c_V2^{-3\lvl(K)}&\le |K|\le C_V2^{-3\lvl(K)},\label{eq:volscale}\\
 c_h2^{-\lvl(K)}&\le \diam(K)\le C_h2^{-\lvl(K)}.\label{eq:diamscale}
\end{align}
\end{lemma}

\begin{proof}
Let $K_0\in\mathcal{T}_0$ be the ancestor of $K$ and write
$\gen(K)=3q+s$ with $s\in\{0,1,2\}$.  Since each bisection halves volume,
\[
 |K|=2^{-3q-s}|K_0|.
\]
As $\mathcal{T}_0$ is finite, this immediately gives \eqref{eq:volscale}.  Uniform shape
regularity implies $|K|\simeq\diam(K)^3$, with equivalence constants depending
only on the finite family of similarity classes and on $\mathcal{T}_0$.  Combining this
with \eqref{eq:volscale} gives \eqref{eq:diamscale}.
\end{proof}

For a vertex $z\in\mathcal{V}(\mathcal{U}_q)$ and an edge
$e\in\mathcal{E}(\mathcal{U}_q)$, set
\[
 \omega_q(z):=\{K\in\mathcal{U}_q:z\in\mathcal{V}(K)\},
 \qquad
 \omega_q(e):=\{K\in\mathcal{U}_q:e\in\mathcal{E}(K)\}.
\]

\begin{lemma}[Uniform vertex and edge valence]\label{lem:valence}
The quantity
\[
 \nu=\nu(\mathcal{T}_0)
 :=\sup_{q\ge0}\sup_{z\in\mathcal{V}(\mathcal{U}_q)}\card\omega_q(z)<\infty.
\]
Moreover, it holds that
\[
 \card\omega_q(e)\le\nu
 \qquad\text{for every }q\ge0,\ e\in\mathcal{E}(\mathcal{U}_q).
\]
\end{lemma}

\begin{proof}
The AMP finite-similarity-class theorem gives a positive lower bound, depending
only on $\mathcal{T}_0$, for every solid angle of every descendant tetrahedron.  The
solid-angle sectors of the tetrahedra in a vertex star have disjoint interiors
on the unit sphere and their total area is at most $4\pi$.  This proves the
uniform vertex bound.  For either endpoint $z$ of an edge $e$,
$\omega_q(e)\subset\omega_q(z)$, which proves the edge bound.
\end{proof}

We first isolate a direct consequence of the face-marking rules.  Besides
shortening the proof of the old-edge property, it will later exclude the second
microgeneration from a different-edge horizontal transition.

\begin{lemma}[Old face marks in the first microgeneration]
\label{lem:firstmicrofacemarks}
Let $K$ be an unflagged marked tetrahedron and let $T$ be a child of $K$.
Every face mark of $T$ belongs to $\mathcal{E}(K)$.  Moreover, if $Q$ is a child of
$T$, then
\begin{equation}\label{eq:depthtwointersection}
 \mathcal{E}(K)\cap\mathcal{E}(Q)=\{\refedge(Q)\}.
\end{equation}
\end{lemma}

\begin{proof}
An inherited face of $T$ keeps a marked edge of $K$.  A cut face is marked by
the edge opposite the new midpoint, which is an uncut edge of $K$.  Finally,
because $K$ is unflagged and hence is not of type $P_f$, the new face is marked
by the edge opposite the new midpoint, again an edge of $K$.  Thus every face
mark of $T$ is an edge of $K$.

The refinement edge of $T$ is the mark of its inherited face and is therefore
an edge of $K$.  The tetrahedron $T$ has three vertices of $K$ and the midpoint
created in the bisection of $K$.  Bisecting $T$ along its refinement edge leaves
each child $Q$ with exactly two vertices of $K$, and hence with at most one edge
of $K$.  On the other hand, $\refedge(Q)$ is the mark of its inherited face;
the first part shows that this is an edge of $K$.  This proves
\eqref{eq:depthtwointersection}.
\end{proof}

The following observation is the basic combinatorial fact behind the causal
reduction.

\begin{lemma}[Macro-generation of the refinement edge]\label{lem:redge}
Let $K$ be an AMP descendant with
\[
  \gen(K)=3q+s,\qquad s\in\{0,1,2\},
\]
and let $Q\in\mathcal{U}_q$ be its generation-$3q$ macro-ancestor.  Then
\begin{equation}\label{eq:redgeancestor}
  \refedge(K)\in\mathcal{E}(Q),
  \qquad
  \mgen(\refedge(K))=q=\lvl(K).
\end{equation}
\end{lemma}

\begin{proof}
For $s=0$, one has $K=Q$, so every edge of $K$ is an edge of $Q$ and has
macro-generation $q$.

Let now $s=1$.  Then $K$ is a child of $Q$.  By the AMP bisection
rule, the refinement edge of a child is the marked edge of its inherited face.
That edge is an edge of $Q$, hence has macro-generation $q$.

For $s=2$, let $T$ be the generation-$3q+1$ parent of $K$.  By item~\ref{A1},
$Q$ is unflagged, and \Cref{lem:firstmicrofacemarks} therefore shows that the
refinement edge of $K$ is an edge of $Q$.

This proves the first assertion in \eqref{eq:redgeancestor}.  Since $Q$
belongs to $\mathcal{U}_q$, its edges have macro-generation $q$, which proves the
second assertion.
\end{proof}

\begin{remark}
The proof uses exactly the feature that AMP organizes refinement into cycles of
three bisections.  New macro-edges of generation $q+1$ may already occur in
generation $3q+1$ or $3q+2$ tetrahedra, but these new edges are not used as the
refinement edge until the next macrostep.
\end{remark}

\subsection{Causal refinement chains and monotonicity}

\subsubsection{The intrinsic refinement tree of a marked face}

For a marked triangular face $F$, denote its marked edge by $\mu_F$.  Every
triangular face $G\subseteq F$ produced by bisections of tetrahedra incident
to $F$ is called a \emph{trace triangle}.  At any stage, the current trace
triangles form the triangulation induced on $F$ from that side.  A
hanging vertex in the relative interior of $F$ need not itself be the midpoint
of an edge of the event tetrahedron.  We therefore trace the induced face
refinement back to its first split.

\begin{lemma}[Intrinsic face-refinement tree]\label{lem:facetree}
Let $F$ be a face in a conformingly marked AMP mesh.  The trace triangles
induced on $F$ from either side form the same intrinsic binary marked-face
refinement tree $\mathbb F(F)$ rooted at $F$.  An ambient tetrahedral
bisection either leaves a current trace triangle $G$ unchanged with the same
mark, or splits $G$ at the midpoint of its marked edge.  Consequently, every
nontrivial trace refinement of $F$ contains
\[
  b_F:=\operatorname{mid}(\mu_F)
\]
and $b_F$ is created before every vertex introduced at a proper descendant of
the root.
\end{lemma}

\begin{proof}
Let $K$ be a marked tetrahedron having a current trace triangle $G$ as a face.
There are two cases.
If $\refedge(K)\subset G$, then $G$ is a refinement face and, by the
definition of a marked tetrahedron, its marked edge is
$\mu_G=\refedge(K)$.  The bisection of $K$ therefore introduces
$\operatorname{mid}(\mu_G)$ and cuts $G$ into two triangles.  AMP's rule for cut
faces marks each of these triangles by its edge opposite the new midpoint;
the result depends only on the marked triangle $(G,\mu_G)$.

If $\refedge(K)\not\subset G$, then exactly one child of $K$ inherits $G$.
No new vertex is introduced on $G$, and AMP's inherited-face rule preserves
$\mu_G$; moreover that inherited mark is the refinement edge of the child.
Thus, if the branch containing $G$ is bisected again, its first bisection that
changes the trace on $G$ necessarily bisects $\mu_G$.

The two alternatives recursively define a binary marked-face tree depending
only on the marked triangle $(G,\mu_G)$.  The two sides start from the same
marked root $F$; induction over successive splits therefore gives the same
descendant triangles and marks from both sides.  A nontrivial finite subtree
splits its root before a proper descendant can occur, which proves the final
assertion.
\end{proof}

\subsubsection{A causal forest for one call of \texttt{\emph{RefineAMP}}}

We now organize the bisections generated by one call of $\texttt{RefineAMP}$ into a
causal forest.  Fix a conformingly marked AMP mesh $\mathcal{T}$, a marking set $\M\subseteq\mathcal{T}$,
and one call
\[
 \mathcal{T}^+=\texttt{RefineAMP}(\mathcal{T},\M)
\]
and write the calculation as successive batches.  Batch zero bisects the marked
elements $\M$.  Each later batch bisects all active tetrahedra having a hanging
vertex in the intermediate mesh produced by the preceding batches.  Regard
every actual bisection as an \emph{event}; the batch-zero events will be the
roots.

\begin{lemma}[Current-edge causal predecessor]\label{lem:causaledge}
Let $K$ be the tetrahedron bisected by a non-root event.  Immediately before
that event there exist an edge $e\in\mathcal{E}(K)$ and a hanging vertex
$b=\operatorname{mid}(e)$ of $K$ such that an earlier event in the same call
bisected the geometric edge $e$ and created $b$.
\end{lemma}

\begin{proof}
Let $\widetilde{\mathcal T}$ be the intermediate mesh immediately before the batch
containing the event $K$.  Since $K$ is a non-root event tetrahedron, it has a
hanging vertex
$z\in\mathcal{V}(\widetilde{\mathcal T})\cap(\partial K\setminus\mathcal{V}(K))$.
By definition, $z$ lies in the relative interior of an edge or a face of $K$.

If $z$ lies in the relative interior of an edge $e_0\in\mathcal{E}(K)$, the
subdivision induced on $e_0$ is obtained recursively by midpoint bisection and
hence belongs to the dyadic binary tree rooted at $e_0$.  Since $z$ is one of
its interior vertices, this subdivision is proper, and its first split
introduces
\[
  b:=\operatorname{mid}(e_0).
\]
Consequently $b$ is a current mesh vertex but is not a vertex of $K$, and we
take $e=e_0$.

Otherwise, $z$ lies in the relative interior of a face $F$ of $K$.  This cannot
be an exposed boundary face, so the trace from the other side is a proper
refinement of $F$.  By the face rules \textnormal{(B1)}--\textnormal{(B3)},
tracing $F$ backward through its face ancestry reaches either a shared face
of the input mesh $\mathcal{T}$ or a common new face created by an earlier
bisection.  In the former case the two incident input tetrahedra induce the
same mark, while in the latter case \textnormal{(B3)} assigns the same mark
from both children.  Hence the argument of \Cref{lem:facetree} applies to the
corresponding intrinsic face tree and, in particular, to its subtree rooted
at $F$.  Its first split therefore introduces
\[
  b:=\operatorname{mid}(\mu_F)
\]
and bisects $\mu_F$.  Thus $b$ is a current mesh vertex but not a vertex of
$K$, and in this case we take $e=\mu_F\in\mathcal{E}(K)$.

In both cases $b$ is the first-split vertex below the trace-tree node $e$.
By the construction of the edge and face trees, that split is realized only
when an incident tetrahedron is bisected along exactly $e$.  Hence some event
that created $b$ bisected $e$.

The creator belongs to the present call.  Indeed, let $K^0\in\mathcal{T}$ be the input
ancestor of $K$.  If $b$ were already a vertex of $\mathcal{T}$, then
$b\in K\subseteq K^0$.  Choose $T\in\mathcal{T}$ having $b$ as a vertex.  If
$T=K^0$, there is nothing to prove; otherwise conformity implies that
$T\cap K^0$ is a common vertex, edge, or face containing $b$, and hence
$b\in\mathcal{V}(K^0)$.  Every child that contains a vertex of its parent retains
that vertex.  Induction from $K^0$ to $K$ would therefore give
$b\in\mathcal{V}(K)$, a contradiction.  Since $b$ is present before the current
batch, it was created in batch zero or in an earlier conformity batch.
\end{proof}

\begin{definition}[Causal forest and causal chains]\label{def:causalforest}
For every non-root event choose one pair $(e,b)$ and one earlier creator event
provided by \Cref{lem:causaledge}, and declare that creator to be its parent.
The resulting directed graph is the \emph{causal forest} of the call.  Its roots
are the batch-zero events, and a \emph{causal chain} is a directed path in this
forest.  The forest depends on the choices of creator events and need not be
canonical.
\end{definition}

Each parent precedes its child, so the directed graph has no cycle; every
non-root event has exactly one chosen parent.  Let $K$ and $T$ be the
tetrahedra bisected by a parent event and its causal child, respectively, and
set
\[
 e_1:=\refedge(K),
 \qquad
 e_2:=\refedge(T).
\]
Then $e_1\in\mathcal{E}(T)$.

\begin{lemma}[One-step macro-generation monotonicity]\label{lem:monotone}
Let a parent event and its causal child bisect $K$ and $T$, respectively, and
set $e_1:=\refedge(K)$ and $e_2:=\refedge(T)$.  If $\mgen(e_1)=q$, then
\begin{equation*}
  \mgen(e_2)\in\{q,q-1\}.
\end{equation*}
In particular, along every root-to-node causal chain the macro-generations of the
successive bisection edges are nonincreasing and can decrease by at most one at
a time.
\end{lemma}

\begin{proof}
Write $\gen(T)=3j+s$, $s\in\{0,1,2\}$.  By
\Cref{lem:redge},
\[
 \mgen(e_2)=j.
\]
By $e_1\in\mathcal{E}(T)$, $e_1$ is an edge of $T$.  If
$s=0$, all edges of $T$ have macro-generation $j$, so $q=j$.  If $s=1$ or
$s=2$, AMP's edge-generation table says that every edge of $T$ has
macro-generation $j$ or $j+1$.  Hence $q\in\{j,j+1\}$ and therefore
$j\in\{q,q-1\}$.
\end{proof}

Let $K_0,\ldots,K_J$ be the event tetrahedra along a causal chain and set
\[
  q_j:=\mgen(\refedge(K_j)),\qquad j=0,\ldots,J.
\]
A \emph{horizontal block} is a maximal consecutive subsequence
$K_r,\ldots,K_s$ for which
\[
  q_r=q_{r+1}=\cdots=q_s.
\]
By \Cref{lem:monotone}, all occurrences of any fixed macro-generation
$q$ along a causal chain are consecutive.

\subsection{Macro-star counting along causal chains}

The following two consequences of the macro-ancestry and valence bounds will
be used repeatedly in the horizontal analysis.

\begin{lemma}[Counting possible events over macro-elements]\label{lem:eventcount}
Let $q\ge0$ and $\mathcal P\subset\mathcal{U}_q$.  A causal chain in one call of
\texttt{\emph{RefineAMP}} contains at most $7\,\card\mathcal P$ events of generations
$3q$, $3q+1$, or $3q+2$ whose generation-$3q$ ancestors belong to
$\mathcal P$.
\end{lemma}

\begin{proof}
The binary ancestry tree below one $K\in\mathcal{U}_q$ contains respectively
$1$, $2$, and $4$ nodes at depths $0$, $1$, and $2$.  Their total number is
$7$.  During one refinement call a tetrahedron can be bisected only once, so
the events on a causal chain are distinct tetrahedra in these ancestry trees.
\end{proof}

\begin{lemma}[Fixed-edge multiplicity]\label{lem:fixededge}
Fix $q\ge0$ and a geometric edge $e\in\mathcal{E}(\mathcal{U}_q)$.  In one call of
\texttt{\emph{RefineAMP}}, any causal chain contains at most
$7\,\card\omega_q(e)\le7\nu$ events whose bisection edge is exactly $e$.
\end{lemma}

\begin{proof}
Let an event tetrahedron $K$ satisfy $\refedge(K)=e$.  By
\Cref{lem:redge}, $\lvl(K)=\mgen(e)=q$, so
$\gen(K)\in\{3q,3q+1,3q+2\}$.  If $Q\in\mathcal{U}_q$ is the generation-$3q$
macro-ancestor of $K$, the stronger containment
\eqref{eq:redgeancestor} gives $e\in\mathcal{E}(Q)$; hence
$Q\in\omega_q(e)$.  Each such $Q$ has only the seven possible event nodes at
depths zero, one, and two, and an event tetrahedron is bisected at most once in
the call.  Applying \Cref{lem:eventcount} to
$\mathcal P=\omega_q(e)$ proves the claim.
\end{proof}

\section{Horizontal propagation}
\label{sec:horizontal-propagation}

We prove that every horizontal block defined in the preceding section has
uniformly bounded length.  The proof combines the macro-star counting from
the preceding section with a finite classification of the possible edge
transitions within one AMP macrostep.

\subsection{The horizontal-propagation bound}

\begin{proposition}[Horizontal propagation bound]\label{prop:horizontal}
Let $\mathcal{T}_0$ be an initial AMP marked mesh.  For every \texttt{\emph{RefineAMP}} call on a
descendant of $\mathcal{T}_0$, each horizontal block of a causal chain has at most
\begin{equation*}
 H_{\mathrm{AMP}}
 :=
 \max\{448\,\card\mathcal{T}_0,\;630\,\nu\}
\end{equation*}
events, where $\nu$ is the uniform macro-vertex valence from
\Cref{lem:valence}.
\end{proposition}

The proof of \Cref{prop:horizontal} is given in
\Cref{sec:horizontal-proof} after the finite transition analysis below.

By \Cref{lem:fixededge}, it remains to control transitions between distinct
bisection edges in a horizontal block.

\subsection{The five edge classes in one AMP macrostep}

For $q\ge2$, every edge of $\mathcal{U}_q$ is classified relative to $\mathcal{U}_{q-1}$.
Let $K\in\mathcal{U}_{q-1}$ and set $m_{ab}:=\operatorname{mid}([a,b])$ for each edge
$[a,b]$ of $K$.  AMP's description of the $25$ edges created by three uniform
bisections gives the following classes.

\begin{definition}[Macro-edge classes]\label{def:classes}
\begin{enumerate}[label=(\roman*)]
\item $H$ (\emph{half-edge}):\newline
      $[a,m_{ab}]$ or $[b,m_{ab}]$, where $[a,b]\in\mathcal{E}(K)$.
\item $V$ (\emph{vertex-to-midpoint spoke}): if a face $[a,b,c]$ is marked by
      $[a,b]$, the edge $[c,m_{ab}]$.
\item $M_{\mathrm f}$ (\emph{face-medial edge}): if a face $[a,b,c]$ is marked by
      $[a,b]$, one of
      \[
        [m_{ab},m_{ac}],\qquad [m_{ab},m_{bc}].
      \]
\item $I_A$ or $I_P$ (\emph{interior edge}): the unique edge joining the
      midpoint of the refinement edge of $K$ to the midpoint of the opposite
      edge, according as $K$ is of type $A$ or $P_u$.
\end{enumerate}
We write $\operatorname{cls}(e)\in\{H,V,M_{\mathrm f},I_A,I_P\}$ for the class of an edge.
For an $H$- or $V$-edge, $\alpha(e)$ denotes its unique endpoint that is a
vertex of $\mathcal{U}_{q-1}$ and is called its \emph{anchor}.  For an
$M_{\mathrm f}$-edge $e$,
its minimal carrier is a unique face
$\operatorname{car}(e)\in\mathcal{F}(\mathcal{U}_{q-1})$; we define
\begin{equation*}
   \sigma(e):=\text{the marked edge of }\operatorname{car}(e)
   \in\mathcal{E}(\mathcal{U}_{q-1}).
\end{equation*}
\end{definition}

For the finite checks below, work on the labeled tetrahedron $[0,1,2,3]$,
set $\mu_{ijk}:=\mu_{[i,j,k]}$, and denote the canonical type-$A$ and
type-$P_u$ marking patterns by
\[
\begin{array}{c|c|cccc}
 K&\refedge(K)&\mu_{012}&\mu_{013}&\mu_{023}&\mu_{123}\\ \hline
 K_A&[0,3]&[0,2]&[0,3]&[0,3]&[1,3]\\
 K_{P_u}&[0,2]&[0,2]&[0,1]&[0,2]&[1,2]
\end{array}
\]
Both flags are unset.

\begin{lemma}[Canonical AMP reduction]
\label{lem:canonicalmacrocycle}
Up to a permutation of vertex labels, the fully marked patterns of types $A$
and $P_u$ are $K_A$ and $K_{P_u}$, respectively.  Since $\BisectTet$ is
equivariant under such permutations, their depth-three descendants exhaust
all AMP macrochildren up to relabeling.
\end{lemma}

\begin{proof}
For type $P_u$, coplanarity forces the two nonrefinement marks to meet at one
vertex outside the refinement edge, which gives the second row after
relabeling.  For type $A$, they meet the two endpoints of the refinement edge,
and noncoplanarity forces their other endpoints to be distinct, which gives
the first row.  Extend a relabeling $\pi$ by
$\pi(\operatorname{mid}([a,b])):=
\operatorname{mid}([\pi(a),\pi(b)])$.
Rules \textnormal{(B1)}--\textnormal{(B4)} use only face incidences, opposite
edges, transported midpoints, and the flag, and hence commute with relabeling.
Iteration proves the claim.
\end{proof}

Every edge of $\mathcal{U}_q$ has a unique minimal carrier in the conforming mesh
$\mathcal{U}_{q-1}$.  Its dimension distinguishes the half-edge class $H$,
the face-carried classes $V$ and $M_{\mathrm f}$, and the interior classes
$I_A$ and $I_P$,
while within a carrier face the presence of an old endpoint distinguishes
$V$ from $M_{\mathrm f}$.  Since shared faces have the same AMP mark, both
$\operatorname{cls}$ and $\sigma$ are independent of the incident
macro-tetrahedron.  Applying \textnormal{(B1)}--\textnormal{(B4)} for three
generations gives, for either parent type, exactly $12$ edges of class $H$,
$4$ of class $V$, $8$ of class $M_{\mathrm f}$, and one of class $I_A$ or $I_P$.  Their
total is $25$, so \Cref{lem:canonicalmacrocycle} makes the classification
exhaustive.

\subsection{Finite horizontal transition tables}

We first localize every different-edge horizontal transition to one
macro-tetrahedron and its first microgeneration.

\begin{lemma}[Localization of horizontal transitions]
\label{lem:horizontalreduction}
Let $q\ge2$ and let $K\to T$ be consecutive events in a horizontal
macro-generation-$q$ block.  Set
$e_1:=\refedge(K)$ and $e_2:=\refedge(T)$, and suppose that $e_1\ne e_2$.
If $Q\in\mathcal{U}_q$ is the generation-$3q$ macro-ancestor of $T$, then
\[
 e_1,e_2\in\mathcal{E}(Q)\cap\mathcal{E}(T),
 \qquad
 \gen(T)\in\{3q,3q+1\}.
\]
\end{lemma}

\begin{proof}
Since the two events belong to a horizontal macro-generation-$q$ block,
\[
 \mgen(e_1)=\mgen(e_2)=q.
\]
\Cref{lem:redge} therefore gives
$\lvl(K)=\lvl(T)=q$, so $T$ has the stated ancestor $Q\in\mathcal{U}_q$.
The causal relation gives $e_1\in\mathcal{E}(T)$, while
$\mgen(e_1)=q$ gives $e_1\in\mathcal{E}(\mathcal{U}_q)$.  Choose $Q_1\in\mathcal{U}_q$ with
$e_1\in\mathcal{E}(Q_1)$.  If $Q_1=Q$, there is nothing to prove.  Otherwise
\[
e_1\subset Q_1\cap Q,
\]
and conformity of $\mathcal{U}_q$ implies that $Q_1\cap Q$ is a common edge or face.
In either case $e_1\in\mathcal{E}(Q)$.

Applying \Cref{lem:redge} to $T$ gives
$e_2\in\mathcal{E}(Q)$, while $e_2\in\mathcal{E}(T)$ by definition.
This proves the first assertion.

Write
\[
\gen(T)=3q+s,
\qquad
s\in\{0,1,2\}.
\]
If $s=2$, $Q$ is unflagged by item~\ref{A1}, and
\Cref{lem:firstmicrofacemarks} gives
\[
\mathcal{E}(Q)\cap\mathcal{E}(T)=\{e_2\}.
\]
Together with the first assertion, this implies
$e_1=e_2$, contrary to the assumption.  Hence $s\in\{0,1\}$, which
proves the second assertion.
\end{proof}

\subsubsection{The full class table}

The checks below use one common finite enumeration.  Starting from the
representatives $K_A$ and $K_{P_u}$, apply $\BisectTet$ uniformly for three
generations.
This gives eight fully marked macrochildren $Q$ for each parent.  For every
$Q$, inspect the three possible
targets consisting of $Q$ itself and its two children.  For each target $T$,
set $e_2:=\refedge(T)$ and let $e_1$ range over
\[
 e_1\in\bigl(\mathcal{E}(Q)\cap\mathcal{E}(T)\bigr)\setminus\{e_2\}.
\]
Thus there are $2\cdot8\cdot3=48$ target configurations, and every candidate
edge $e_1$ in each configuration is inspected.  In addition to the class and
anchor of each edge pair, we record its $\sigma$-pair when both edges have
class $M_{\mathrm f}$.
For $q\ge2$, item~\ref{A1} shows that the macro-parent in $\mathcal{U}_{q-1}$ is of
type $A$ or $P_u$.  Hence
\Cref{lem:canonicalmacrocycle,lem:horizontalreduction} show that this check
contains every different-edge horizontal transition up to relabeling.
Relabeling preserves classes and transports anchors and $\sigma$-edges, so
the recorded data are exhaustive as well.

\begin{lemma}[Horizontal class transitions]\label{lem:classtable}
Let $q\ge2$.  Suppose that two consecutive events in a horizontal
macro-generation-$q$ block bisect distinct edges $e_1$ and $e_2$.  Then their
ordered class pair is one of those indicated in
\begin{equation}\label{eq:classpairs}
\begin{array}{c|ccccc}
 \operatorname{cls}(e_1)\backslash\operatorname{cls}(e_2)
   &H&V&M_{\mathrm f}&I_A&I_P\\ \hline
 H   &\bullet&\bullet&&&\\
 V   &&\bullet&&&\\
 M_{\mathrm f}   &\bullet&\bullet&\bullet&\bullet&\\
 I_A &&\bullet&&&\\
 I_P &&\bullet&\bullet&&
\end{array}
\end{equation}
Moreover, each of the transitions
\[
 H\to H,\qquad H\to V,\qquad V\to V,
\]
preserves the anchor:
\[
  \alpha(e_1)=\alpha(e_2).
\]
\end{lemma}

\begin{proof}
Applying the enumeration above to a type-$A$ parent gives
\[
 \{H\!\to\!H,\ H\!\to\!V,\ V\!\to\!V,\ M_{\mathrm f}\!\to\!H,
   \ M_{\mathrm f}\!\to\!V,\ M_{\mathrm f}\!\to\!M_{\mathrm f},
   \ M_{\mathrm f}\!\to\!I_A,\ I_A\!\to\!V\}.
\]
For a type-$P_u$ parent the same enumeration gives
\[
 \{H\!\to\!H,\ H\!\to\!V,\ V\!\to\!V,\ M_{\mathrm f}\!\to\!H,
   \ M_{\mathrm f}\!\to\!V,\ M_{\mathrm f}\!\to\!M_{\mathrm f},
   \ I_P\!\to\!M_{\mathrm f},\ I_P\!\to\!V\}.
\]
Their union is exactly \eqref{eq:classpairs}.  Inspecting the corresponding
edge pairs before passing to classes shows that every $H\to H$, $H\to V$, or
$V\to V$ pair has the same old endpoint on both sides.  This proves the
anchor assertion.  The completeness statement preceding the lemma shows that
the table is exhaustive.
\end{proof}

Ignoring the three self-arrows $H\to H$, $V\to V$, and
$M_{\mathrm f}\to M_{\mathrm f}$, the class
graph is acyclic:
\[
I_P\longrightarrow M_{\mathrm f}\longrightarrow
\begin{cases}
H\longrightarrow V,\\
I_A\longrightarrow V,\\
V,
\end{cases}
\qquad
I_P\longrightarrow V.
\]
Thus only an $M_{\mathrm f}\to M_{\mathrm f}$ block requires more than a
vertex-patch argument.

\subsubsection{The spine relation}

For a type-$A$ or type-$P_u$ macro-tetrahedron $K$, let
$\mu_1(K),\mu_2(K)$ denote the marked edges of its two nonrefinement faces.
Define the two \emph{spine relations}
\begin{equation*}
   \mu_i(K)\rightsquigarrow\refedge(K),
   \qquad i\in\{1,2\}.
\end{equation*}
In canonical labels they are
\[
 \begin{array}{c|c}
 K_A&[0,2]\rightsquigarrow[0,3],\quad
      [1,3]\rightsquigarrow[0,3],\\
 K_{P_u}&[0,1]\rightsquigarrow[0,2],\quad
          [1,2]\rightsquigarrow[0,2].
 \end{array}
\]

\begin{lemma}[Spine class transitions]\label{lem:spinetable}
Let $p\ge2$, let $K\in\mathcal{U}_p$, and let $i\in\{1,2\}$.  With classes taken
relative to $\mathcal{U}_{p-1}$, the possible ordered pairs
$(\operatorname{cls}(\mu_i(K)),\operatorname{cls}(\refedge(K)))$ are exactly
those indicated in
\[
\begin{array}{c|ccccc}
 \operatorname{cls}(\mu_i(K))\backslash\operatorname{cls}(\refedge(K))
   &H&V&M_{\mathrm f}&I_A&I_P\\ \hline
 H   &\bullet&\bullet&&&\\
 V   &&\bullet&&&\\
 M_{\mathrm f}   &\bullet&\bullet&&&\\
 I_A &&\bullet&&&\\
 I_P &&&&&
\end{array}
\]
For $H\to H$, $H\to V$, and $V\to V$, the anchor is preserved:
\[
 \alpha(\mu_i(K))=\alpha(\refedge(K)).
\]
\end{lemma}

\begin{proof}
Inspecting the two spine relations of every macrochild in the common finite
enumeration gives, for a type-$A$ parent,
\[
 \{H\!\to\!H,\ H\!\to\!V,\ V\!\to\!V,
   \ M_{\mathrm f}\!\to\!H,\ M_{\mathrm f}\!\to\!V,
   \ I_A\!\to\!V\},
\]
while for a type-$P_u$ parent they are
\[
 \{H\!\to\!H,\ H\!\to\!V,\ V\!\to\!V,
   \ M_{\mathrm f}\!\to\!H,\ M_{\mathrm f}\!\to\!V\}.
\]
This gives the stated table.  The anchor assertion follows directly from the
recorded edge pairs.  The parent in $\mathcal{U}_{p-1}$ has type $A$ or $P_u$ by
item~\ref{A1}, so \Cref{lem:canonicalmacrocycle} makes the enumeration
exhaustive.
\end{proof}

\subsubsection{Projection of medial transitions}

\begin{lemma}[Medial-to-spine projection]\label{lem:Mprojection}
Let $q\ge2$, and let $e_1\to e_2$ be a different-edge horizontal transition
with $\operatorname{cls}(e_1)=\operatorname{cls}(e_2)=M_{\mathrm f}$.  Then either
$\sigma(e_1)=\sigma(e_2)$, or
$\sigma(e_1)\rightsquigarrow\sigma(e_2)$ is one of the two spine relations of a
tetrahedron $K\in\mathcal{U}_{q-1}$ containing the transition.
\end{lemma}

\begin{proof}
Restrict the common edge-level enumeration to the pairs
$M_{\mathrm f}\to M_{\mathrm f}$ and apply
$\sigma$ to both entries.  The resulting pairs are
\begin{equation*}
\begin{array}{c|c}
 K_A&
 (\sigma(e_1),\sigma(e_2))\in
 \{([0,2],[0,3]),([1,3],[0,3]),([0,3],[0,3])\},\\[0.2em]
 K_{P_u}&
 (\sigma(e_1),\sigma(e_2))\in
 \{([0,1],[0,2]),([1,2],[0,2]),([0,2],[0,2])\}.
\end{array}
\end{equation*}
The unequal pairs in the first row are precisely the two type-$A$ spine
relations, and the unequal pairs in the second row are precisely the two
type-$P_u$ spine relations.  The remaining pair in each row is equality.
The completeness of the common enumeration proves the assertion for every
different-edge horizontal transition.
\end{proof}

\subsection{Proof of the horizontal-propagation bound}
\label{sec:horizontal-proof}

\begin{proof}[Proof of \Cref{prop:horizontal}]
Fix a horizontal macro-generation-$q$ block
\[
 K_0,K_1,\ldots,K_J
\]
of events and write $e_j=\refedge(K_j)$.  The event tetrahedra are distinct
nodes of the binary ancestry trees generated by bisection.

\medskip
\noindent\emph{Low macro-generations.}
For a fixed initial tetrahedron, the total number of possible event
tetrahedra at generations $3q$, $3q+1$, and $3q+2$ is
\[
  2^{3q}+2^{3q+1}+2^{3q+2}=7\,8^q.
\]
Therefore, for $q=0,1,2$,
\[
  J+1\le 7\,8^q\card\mathcal{T}_0\le448\,\card\mathcal{T}_0.
\]
It remains to consider $q\ge3$.  In this range, the macro-meshes
$\mathcal{U}_{q-2}$, $\mathcal{U}_{q-1}$, and $\mathcal{U}_q$ consist entirely of persistent types
$A$ and $P_u$, so all preceding transition tables apply.

\medskip
\noindent\emph{Decomposition by edge class.}
If $e_j=e_{j+1}$, their classes agree.  If they are distinct,
\Cref{lem:classtable} applies.  The directed class graph therefore shows that
the causal chain contains at most one block of each of
\[
 I_P,\quad M_{\mathrm f},\quad I_A,\quad H,\quad V.
\]
Some blocks may be absent, and their order is constrained by
\eqref{eq:classpairs}.

\medskip
\noindent\emph{Interior blocks.}
There is no different-edge transition $I_P\to I_P$ or $I_A\to I_A$.
Consequently, within either interior block all bisection edges are the same
geometric edge $e\in\mathcal{E}(\mathcal{U}_q)$.  The fixed-edge multiplicity estimate
in \Cref{lem:fixededge} therefore shows that each interior block has at most
\begin{equation*}
  7\nu
\end{equation*}
events.

\medskip
\noindent\emph{Half-edge and spoke blocks.}
For an $H$-block, equality of consecutive edges preserves the anchor, and every
different-edge $H\to H$ transition preserves the anchor by
\Cref{lem:classtable}.  Hence all bisection edges in the block have one common
anchor $a\in\mathcal{V}(\mathcal{U}_{q-1})$.  Every macro-tetrahedron of $\mathcal{U}_q$ containing one
of those edges descends from a tetrahedron in $\omega_{q-1}(a)$.  There are at
most $8\nu$ such macro-tetrahedra of $\mathcal{U}_q$.  By
\Cref{lem:eventcount}, the block has at most
\begin{equation*}
  7\cdot8\nu=56\nu
\end{equation*}
events.  The same proof, using $V\to V$ and anchor preservation, gives
\begin{equation*}
  \card(\text{$V$-block})\le56\nu.
\end{equation*}

\medskip
\noindent\emph{The medial block.}
Let
\[
 e_{j_0},e_{j_0+1},\ldots,e_{j_1}
\]
be the $M_{\mathrm f}$-block and put
$\mu_j=\sigma(e_j)\in\mathcal{E}(\mathcal{U}_{q-1})$.  By
\Cref{lem:Mprojection}, each consecutive pair satisfies either
$\mu_{j+1}=\mu_j$ or
\[
  \mu_j\rightsquigarrow\mu_{j+1}
\]
as a spine relation in $\mathcal{U}_{q-1}$.

If every $\mu_j$ is equal to one fixed edge $\mu$, then the carrier face of
every $e_j$ is marked by $\mu$ and therefore contains $\mu$.  Since $\mathcal{U}_{q-1}$ is
conforming and
$\operatorname{relint}(e_j)\subset
 \operatorname{relint}(\operatorname{car}(e_j))$, the $\mathcal{U}_{q-1}$-parent of
every $\mathcal{U}_q$-tetrahedron containing $e_j$ contains
$\operatorname{car}(e_j)$.  All relevant
$\mathcal{U}_q$-macro-tetrahedra descend from the edge star $\omega_{q-1}(\mu)$.
There are at most $8\nu$ of them, so the whole $M_{\mathrm f}$-block has at most $56\nu$
events.

Otherwise, let $i$ be the first index for which $\mu_i\ne\mu_{i+1}$.  The prefix
through $e_i$ is controlled by the fixed edge $\mu_i$ and therefore has at most
$56\nu$ events, as above.  The nontrivial transition
$\mu_i\rightsquigarrow\mu_{i+1}$ is a spine relation in $\mathcal{U}_{q-1}$.  Since
$q-1\ge2$, \Cref{lem:spinetable}, with classes relative to $\mathcal{U}_{q-2}$,
shows that $\mu_{i+1}$ is of class $H$ or $V$.  Let
$a\in\mathcal{V}(\mathcal{U}_{q-2})$ be its anchor.  Every later equality preserves this edge and
every later nontrivial spine transition starts from an $H$- or $V$-edge;
\Cref{lem:spinetable} shows that its target is again $H$ or $V$ with the same
anchor $a$.  Hence
\begin{equation*}
  a\in\mu_j\qquad\text{for all }j\ge i+1.
\end{equation*}
The carrier face of $e_j$ is marked by $\mu_j$, and therefore contains $a$.
The same carrier argument shows that every event in the tail lies over the
vertex star
$\omega_{q-2}(a)$.  Two macrosteps create at most
$8^2\nu$ tetrahedra of $\mathcal{U}_q$ over that star, and
\Cref{lem:eventcount} gives at most $7\cdot8^2\nu=448\nu$
tail events.  Consequently,
\begin{equation*}
  \card(\text{$M_{\mathrm f}$-block})\le56\nu+448\nu=504\nu.
\end{equation*}

\medskip
Adding the two possible interior blocks, the one medial block, and the
$H$- and $V$-blocks gives
\[
 J+1
 \le 2(7\nu)+504\nu+56\nu+56\nu
 =630\nu.
\]
Combining this with the low-level estimate proves \Cref{prop:horizontal}.
\end{proof}

\section{Closure complexity of AMP refinement}
\label{sec:amp-closure}

We now derive geometric locality and the cumulative closure estimate from
\Cref{prop:horizontal}.  Throughout this section, $H_{\mathrm{AMP}}$ denotes
the fixed constant in \Cref{prop:horizontal}.  Since it depends only on
$\mathcal{T}_0$, it is not displayed as a separate argument of subsequent
constants.

For nonempty sets $X,Y\subseteq\mathbb R^3$, a point $x\in\mathbb R^3$, and
$r>0$, write
\begin{align*}
 \dist(X,Y)&:=\inf\{|x'-y|:x'\in X,\ y\in Y\},\\
 B(x,r)&:=\{y\in\mathbb R^3:|y-x|<r\}.
\end{align*}
Thus $B(x,r)$ is the open Euclidean ball of radius $r$ centered at $x$.

\subsection{Geometric locality}

We first record its geometric consequence.

\begin{theorem}[Single-call locality]\label{thm:locality}
Let $\mathcal{T}_0$ be an initial AMP marked mesh, let $\mathcal{T}$ be a
conformingly marked AMP descendant mesh of $\mathcal{T}_0$, let
$\M\subseteq\mathcal{T}$ be a marking set, and let
$\mathcal{T}^+=\texttt{\emph{RefineAMP}}(\mathcal{T},\M)$.  For every
$T\in\mathcal{T}^+\setminus\mathcal{T}$, there is a causal root
$M\in\M$ such that
\begin{align}
 \lvl(T)&\le \lvl(M)+1,\label{eq:levelinc}\\
 \dist(T,M)&\le R_{\mathrm{AMP}}\,2^{-\lvl(T)},\label{eq:distlocal}
\end{align}
where $R_{\mathrm{AMP}}:=4H_{\mathrm{AMP}}C_h$ and $C_h$ is from
\Cref{lem:scale}.
\end{theorem}

\begin{proof}
Let $K_0=M,K_1,\ldots,K_J$
be the causal chain ending at the bisection event whose child is $T$.  Put
\[
  q_j:=\mgen(\refedge(K_j)).
\]
By \Cref{lem:redge}, $q_j=\lvl(K_j)$.  By
\Cref{lem:monotone},
\[
  q_0\ge q_1\ge\cdots\ge q_J,
  \qquad q_{j+1}\in\{q_j,q_j-1\}.
\]
Since $T$ is a child of $K_J$,
\[
  \lvl(T)\le \lvl(K_J)+1=q_J+1.
\]
Therefore
\[
  \lvl(T)\le q_0+1=\lvl(M)+1,
\]
which proves \eqref{eq:levelinc}.

For the distance estimate, the parent and child event tetrahedra intersect
geometrically: the child event contains the edge whose midpoint was created by
the parent event.  Hence consecutive $K_j$ have nonempty intersection as closed
sets.  Also $T\subset K_J$.  Choosing one point in each consecutive
intersection and applying the triangle inequality gives
\[
  \dist(T,M)\le \sum_{j=0}^{J}\diam(K_j).
\]
Using \Cref{lem:scale,prop:horizontal} and the fact that each integer between
$q_J$ and $q_0$ occurs at most $H_{\mathrm{AMP}}$ times,
\begin{align*}
 \dist(T,M)
 &\le C_h\sum_{j=0}^{J}2^{-q_j}
 \le H_{\mathrm{AMP}}C_h\sum_{q=q_J}^{q_0}2^{-q}\\
 &\le 2H_{\mathrm{AMP}}C_h\,2^{-q_J}.
\end{align*}
Since $\lvl(T)\le q_J+1$, we have
$2^{-q_J}\le2\,2^{-\lvl(T)}$, which proves \eqref{eq:distlocal} with the
stated value of $R_{\mathrm{AMP}}$.
\end{proof}

\begin{remark}
No bound on the \emph{total} length of a causal chain is needed.  The chain may
descend through arbitrarily many macro-generations.  The diameters then form a
geometric series.  Only the number of same-scale events must be controlled.
This is the central structural reduction.
\end{remark}

\subsection{A self-contained weighted closure theorem}

Together with the scale estimate, \Cref{thm:locality} yields the full
closure-complexity bound.  We give the weighted packing argument in full.

\begin{lemma}[Level-wise packing]\label{lem:packing}
For every $R\ge0$, every conformingly marked AMP descendant mesh $\mathcal{T}$, every
$x\in\mathbb R^3$, and every $k\ge0$,
\begin{equation}\label{eq:Cpackexplicit}
\begin{aligned}
 \mathcal N_k(x,R)
 &:=\{K\in\mathcal{T}:\lvl(K)=k,\ K\cap B(x,R2^{-k})\ne\emptyset\},\\
 \card\mathcal N_k(x,R)
 &\le C_{\mathrm{pack}}(R):=\frac{\kappa_3(R+C_h)^3}{c_V},
 \qquad \kappa_3:=\frac{4\pi}{3}.
\end{aligned}
\end{equation}
\end{lemma}

\begin{proof}
If $K\in\mathcal N_k(x,R)$, then
\Cref{lem:scale} implies
\[
  K\subset B(x,(R+C_h)2^{-k}).
\]
The interiors of the tetrahedra of a conforming mesh are disjoint, while
$|K|\ge c_V2^{-3k}$.  Comparing the sum of their volumes with the volume of the
expanded ball gives
\[
 \card\mathcal N_k(x,R)\,c_V2^{-3k}
 \le \frac{4\pi}{3}(R+C_h)^3 2^{-3k}.
\]
This proves the claim.
\end{proof}

\subsubsection{Causal witnesses}

Fix an initial AMP marked mesh $\mathcal{T}_0$.  A finite \emph{AMP adaptive sequence}
issued from $\mathcal{T}_0$ is a sequence with $\M_\ell\subseteq\mathcal{T}_\ell$ and
\begin{equation*}
\mathcal{T}_{\ell+1}=\texttt{RefineAMP}(\mathcal{T}_\ell,\M_\ell),
 \qquad \ell=0,\ldots,L-1.
\end{equation*}
Let $\mathfrak M:=\bigsqcup_{\ell=0}^{L-1}\M_\ell$
be the disjoint union of all marking sets.  Thus
\begin{equation}\label{eq:numbermarks}
  \card\mathfrak M=\sum_{\ell=0}^{L-1}\card\M_\ell.
\end{equation}

For each $\ell=0,\ldots,L-1$ and every
$T\in\mathcal{T}_{\ell+1}\setminus\mathcal{T}_\ell$,
choose once and for all a \emph{witness element} $M$ (the causal root) supplied by
\Cref{thm:locality}, satisfying
\eqref{eq:levelinc}--\eqref{eq:distlocal}.

Every $T\in\mathcal{T}_L\setminus\mathcal{T}_0$ has a finite witness history.  Starting from
its witness element, continue backward whenever the current witness element is not in
$\mathcal{T}_0$.  Creation times strictly decrease, so the process terminates.  Thus,
for some $N\ge1$, we obtain witness elements
\begin{equation}\label{eq:witnesshistory}
 M_0,M_1,\ldots,M_{N-1}\in\mathfrak M
\end{equation}
with $M_0\in\mathcal{T}_0$, where $M_{j+1}$ was created by a refinement rooted at $M_j$
and $T$ was created by a refinement rooted at $M_{N-1}$.  Hence
\begin{subequations}
\begin{align}
 \lvl(M_{j+1})&\le\lvl(M_j)+1,\label{eq:histlevel}\qquad j=0,\ldots,N-2,\\
 \dist(M_{j+1},M_j)&\le R_{\mathrm{AMP}}2^{-\lvl(M_{j+1})},\label{eq:histdist}\qquad j=0,\ldots,N-2,\\
 \lvl(T)&\le\lvl(M_{N-1})+1,\label{eq:lastlevel}\\
 \dist(T,M_{N-1})&\le R_{\mathrm{AMP}}2^{-\lvl(T)}.\label{eq:lastdist}
\end{align}
\end{subequations}

\subsubsection{Weights}

Set
\begin{equation*}
  F:=R_{\mathrm{AMP}}+C_h.
\end{equation*}
Fix $x\in(0,F]$ and put
\begin{equation*}
  E:=R_{\mathrm{AMP}}+x.
\end{equation*}
For $T\in\mathcal{T}_L\setminus\mathcal{T}_0$ and $M\in\mathfrak M$, define
\begin{equation}\label{eq:lambda}
 \lambda(T,M):=
 \begin{cases}
 F\,2^{\lvl(T)-\lvl(M)},
 &\begin{array}{l}
   \dist(T,M)\le E2^{-\lvl(T)},\\[-0.1em]
   \lvl(T)\le\lvl(M)+1,
  \end{array}\\[0.8em]
 0,&\text{otherwise.}
 \end{cases}
\end{equation}

\begin{lemma}[Uniform upper weight per marked element]\label{lem:upperweight}
For every $M\in\mathfrak M$,
\[
  \sum_{T\in\mathcal{T}_L\setminus\mathcal{T}_0}\lambda(T,M)
  \le 4F\,\frac{\kappa_3(E+3C_h)^3}{c_V}.
\]
\end{lemma}

\begin{proof}
Fix $M$ and write $m=\lvl(M)$.  If $\lambda(T,M)\ne0$ and $\lvl(T)=k$, then
$k\le m+1$ and $\dist(T,M)\le E2^{-k}$.  Choose a fixed point $y\in M$.
Because $\diam(M)\le C_h2^{-m}\le2C_h2^{-k}$ and
$\diam(T)\le C_h2^{-k}$, every such $T$ meets
$B(y,(E+2C_h)2^{-k})$.  Hence \Cref{lem:packing} and
\eqref{eq:Cpackexplicit} bound their number by
\[
  P(E):=\frac{\kappa_3(E+3C_h)^3}{c_V},
\]
and therefore
\begin{align*}
 \sum_{T\in\mathcal{T}_L\setminus\mathcal{T}_0}\lambda(T,M)
 &\le P(E)F\sum_{k=0}^{m+1}2^{k-m}
 \le 4P(E)F.
\end{align*}
This is the stated bound.
\end{proof}

\begin{lemma}[Uniform lower weight]\label{lem:lowerweight}
Every $T\in\mathcal{T}_L\setminus\mathcal{T}_0$ satisfies
\begin{equation}\label{eq:lowerweight}
  \sum_{M\in\mathfrak M}\lambda(T,M)\ge x.
\end{equation}
\end{lemma}

\begin{proof}
Let $M_0,\ldots,M_{N-1}$ be a witness history
\eqref{eq:witnesshistory}, put $\ell=\lvl(T)$, and write
$m_j=\lvl(M_j)$.  Since $M_0\in\mathcal{T}_0$, $m_0=0\le\ell$.  Let
\[
 s:=\max\{j\in\{0,\ldots,N-1\}:m_j\le\ell\}.
\]
The set is nonempty.  By maximality of $s$ and the one-step increase
\eqref{eq:histlevel}, one necessarily has
\begin{equation}\label{eq:ms}
  m_s\in\{\ell-1,\ell\}.
\end{equation}
Indeed, if $s<N-1$ and $m_s\le\ell-2$, then
$m_{s+1}\le m_s+1\le\ell-1$, contradicting maximality; if $s=N-1$, then
\eqref{eq:lastlevel} gives the same conclusion.

If $\dist(T,M_s)\le E2^{-\ell}$, then \eqref{eq:ms} implies
\[
  \lambda(T,M_s)=F2^{\ell-m_s}\ge F,
\]
and we are done because $x\le F$.

Assume therefore that $M_s$ lies outside the $E2^{-\ell}$ neighborhood of $T$.
By \eqref{eq:lastdist} and $E>R_{\mathrm{AMP}}$, the last witness element
$M_{N-1}$ lies inside that
neighborhood.  Hence there is a largest index
\[
 k\in\{s,\ldots,N-2\}
 \quad\text{such that}\quad
 \dist(T,M_k)>E2^{-\ell}.
\]
Then all $M_j$ with $j>k$ lie inside the neighborhood.  Also $j>k\ge s$ and the
maximality of $s$ imply $m_j>\ell$, so all these witness elements satisfy the level
condition in \eqref{eq:lambda}.

For closed sets $X,Y,Z$ one has
$\dist(X,Z)\le\dist(X,Y)+\diam(Y)+\dist(Y,Z)$.  Applying this repeatedly along
the witness history, then using \eqref{eq:lastdist},
\eqref{eq:histdist}, and \Cref{lem:scale}, gives
\begin{align*}
 E2^{-\ell}
 &<\dist(T,M_k)\\
 &\le R_{\mathrm{AMP}}2^{-\ell}
    +\sum_{j=k+1}^{N-1}
      \bigl(\diam(M_j)+\dist(M_j,M_{j-1})\bigr)\\
 &\le R_{\mathrm{AMP}}2^{-\ell}
    +(R_{\mathrm{AMP}}+C_h)\sum_{j=k+1}^{N-1}2^{-m_j}\\
 &=R_{\mathrm{AMP}}2^{-\ell}
   +2^{-\ell}\sum_{j=k+1}^{N-1}
       F2^{\ell-m_j}.
\end{align*}
For $j>k$ the corresponding terms are exactly
$\lambda(T,M_j)$.  Since $E-R_{\mathrm{AMP}}=x$, multiplication by $2^\ell$ yields
\[
  \sum_{j=k+1}^{N-1}\lambda(T,M_j)>x.
\]
This proves \eqref{eq:lowerweight}.
\end{proof}

\begin{theorem}[Unconditional AMP closure estimate]\label{thm:unconditional}
Every finite AMP adaptive sequence issued from an initial AMP marked mesh
$\mathcal{T}_0$ satisfies
\begin{equation*}
  \card\mathcal{T}_L-\card\mathcal{T}_0
  \le C_{\mathrm{clos}}(\mathcal{T}_0)
       \sum_{\ell=0}^{L-1}\card\M_\ell.
\end{equation*}
Here the constant $C_{\mathrm{clos}}(\mathcal{T}_0)$ is bounded by
\[
 C_{\mathrm{clos}}(\mathcal{T}_0)
 \le 36\,\frac{\pi C_h^3}{c_V}
 (4H_{\mathrm{AMP}}+1)(4H_{\mathrm{AMP}}+3)^2.
\]
\end{theorem}

\begin{proof}
Fix $x\in(0,F]$.  By \Cref{lem:lowerweight,lem:upperweight},
\begin{align*}
 \card\mathcal{T}_L-\card\mathcal{T}_0
 &\le \card(\mathcal{T}_L\setminus\mathcal{T}_0)\\
 &\le \frac1x
       \sum_{T\in\mathcal{T}_L\setminus\mathcal{T}_0}
       \sum_{M\in\mathfrak M}\lambda(T,M)\\
 &=\frac1x
       \sum_{M\in\mathfrak M}
       \sum_{T\in\mathcal{T}_L\setminus\mathcal{T}_0}\lambda(T,M)\\
 &\le \frac{4F}{x}\,
       \frac{\kappa_3(E+3C_h)^3}{c_V}\card\mathfrak M\\
 &=\frac{4F}{x}\,
       \frac{\kappa_3(R_{\mathrm{AMP}}+x+3C_h)^3}{c_V}
   \sum_{\ell=0}^{L-1}\card\M_\ell,
\end{align*}
where \eqref{eq:numbermarks} was used in the last step.
Thus one may take
\[
 C_{\mathrm{clos}}^{(x)}
 :=\frac{4F}{x}\,
   \frac{\kappa_3(R_{\mathrm{AMP}}+x+3C_h)^3}{c_V}.
\]
Differentiating its logarithm gives the unconstrained minimizer
\[
 x_*=\frac{R_{\mathrm{AMP}}+3C_h}{2}
     =\frac{C_h(4H_{\mathrm{AMP}}+3)}2.
\]
Since $H_{\mathrm{AMP}}\ge1$, one has
$R_{\mathrm{AMP}}=4C_hH_{\mathrm{AMP}}\ge C_h$, which is exactly the condition
$x_*\in(0,F]$.  Substitution gives the stated bound.
\end{proof}

\section{Closure complexity of B\"ansch refinement}
\label{sec:bansch-analysis}

We next return to the face-based formulation introduced by B\"ansch.  The
marking, local bisection rule, red--black transition, and termination and
conformity of the algorithm formulated there are due to
B\"ansch~\cite{Bansch1991}; the local equivalence with
$\BisectTet$ is due to Arnold, Mukherjee, and Pouly
\cite{AMP2000}.  The argument below first isolates the initial
two-edge ambiguities and then transfers arbitrary, possibly
history-dependent executions to AMP\@.  This yields termination and a uniform
closure estimate for all such choices.

\subsection{Intrinsic face marking}

A \emph{B\"ansch face marking} assigns one edge
\[
  \mu_F\in\mathcal{E}(F)
\]
to every triangular face $F$.  For a tetrahedron $K$, define its set of
\emph{global refinement edges} by
\begin{equation}\label{eq:bansch-global-set}
 \mathcal G(K)
 :=
 \{e\in\mathcal{E}(K):\mu_F=e
     \text{ for both faces $F\subset K$ containing $e$}\}.
\end{equation}
B\"ansch's admissibility assumptions are
\begin{enumerate}[label=\textnormal{(F\arabic*)},leftmargin=2.8em]
\item $\mathcal G(K)\ne\varnothing$ for every tetrahedron $K$;
\item if a face $F$ is shared by two tetrahedra, then its marked edge is the
same when viewed from either side.
\end{enumerate}
We use \emph{intrinsic} face marks, attached to geometric faces rather than
to incidences $(K,F)$, so item~\textnormal{(F2)} is built into the notation.
An intrinsic face marking is called \emph{admissible} if it also satisfies
item~\textnormal{(F1)}~\cite[p.~184]{Bansch1991}.

\begin{fixedfigure}
\begin{tikzpicture}[
  faceedge/.style={line width=.38pt},
  facemark/.style={line width=1.15pt},
  every node/.style={font=\footnotesize}
]
\draw[faceedge] (-5.0,0)--(-3.6,0)--(-4.3,1.18)--cycle;
\draw[facemark] (-4.88,.07)--(-3.72,.07);
\node at (-4.3,-.30) {$F\subset K^+$};
\draw[faceedge] (-2.8,0)--(-1.4,0)--(-2.1,1.18)--cycle;
\draw[facemark] (-2.68,.07)--(-1.52,.07);
\node at (-2.1,-.30) {$F\subset K^-$};
\node at (-3.2,.55) {$=$};
\node at (-3.2,-.76) {$\mu_F\text{ is independent of the side}$};

\begin{scope}[xshift=2.0cm]
  \coordinate (ba) at (-.55,0);
  \coordinate (bb) at (.55,0);
  \coordinate (bc) at (0,-.95);
  \coordinate (bdt) at (0,.95);
  \coordinate (bdl) at (-1.10,-.95);
  \coordinate (bdr) at (1.10,-.95);
  \draw[faceedge] (ba)--(bb)--(bc)--cycle;
  \draw[faceedge] (ba)--(bdt)--(bb);
  \draw[faceedge] (ba)--(bdl)--(bc);
  \draw[faceedge] (bb)--(bdr)--(bc);
  \draw[facemark] (-.47,.07)--(.47,.07);
  \draw[facemark] (-.47,-.07)--(.47,-.07);
  \draw[facemark] (-1.00,-.88)--(-.10,-.88);
  \draw[facemark] (.10,-.88)--(1.00,-.88);
  \node[left] at (ba) {$a$};
  \node[right] at (bb) {$b$};
  \node[below] at (bc) {$c$};
  \node[above] at (bdt) {$d$};
  \node[left] at (bdl) {$d$};
  \node[right] at (bdr) {$d$};
  \node at (0,-1.38) {$\mathcal G(K)=\{[a,b],[c,d]\}$};
\end{scope}
\end{tikzpicture}
\caption{Face compatibility and the two-edge ambiguity in the B\"ansch
marking.  On the left, a shared face carries the same intrinsic mark from
both incident tetrahedra.  On the right, the unfolded copies show that $[a,b]$
and $[c,d]$ each mark both faces containing that edge, so both are global
refinement edges; they are opposite.}
\label{fig:banschmarking}
\end{fixedfigure}

\begin{lemma}[The two-edge alternative]\label{lem:bansch-two-global}
An admissibly face-marked tetrahedron has either one or two global refinement
edges.  If it has two, they are opposite.
\end{lemma}

\begin{proof}
A triangular face has only one mark, so two distinct global edges cannot lie
in one face.  Two edges of a tetrahedron fail to share a face exactly when
they are opposite.  The four faces incident to an opposite pair exhaust the
faces of the tetrahedron, which also rules out a third global edge.
\end{proof}

Fix a global refinement edge $[a,b]\in\mathcal G(K)$, write the remaining
vertices as $c,d$, and set
\[
 \mu_a:=\mu_{[a,c,d]},
 \qquad
 \mu_b:=\mu_{[b,c,d]}.
\]
Since the two faces containing $[a,b]$ are marked by $[a,b]$, the remaining
configuration is determined by $(\mu_a,\mu_b)$.  B\"ansch calls $K$
\emph{red} or \emph{black}, respectively, when
\[
\begin{aligned}
 K\text{ is red}
 &\quad\Longleftrightarrow\quad
 (\mu_a,\mu_b)\in
 \{([a,c],[b,d]),([a,d],[b,c])\},\\
 K\text{ is black}
 &\quad\Longleftrightarrow\quad
 (\mu_a,\mu_b)\in
 \{([a,c],[b,c]),([a,d],[b,d])\}.
\end{aligned}
\]
In a black configuration there is therefore a unique $w\in\{c,d\}$ such
that $(\mu_a,\mu_b)=([a,w],[b,w])$.  In the two remaining configurations
exactly one of $\mu_a,\mu_b$ equals $[c,d]$, or both do.  B\"ansch leaves these two
configurations unnamed.

\begin{fixedfigure}
\begin{tikzpicture}[
  faceedge/.style={line width=.38pt},
  facemark/.style={line width=1.15pt,line cap=round},
  tetedge/.style={line width=.45pt,line join=round},
  tethidden/.style={densely dashed,line width=.40pt},
  tetmark/.style={line width=1.15pt,line cap=round},
  dualmark/.style={double=white,double distance=.75pt,line width=.42pt,
                   line cap=round},
  vertex/.style={font=\scriptsize},
  every node/.style={font=\footnotesize}
]
\newcommand{\BanschTetra}{%
  \coordinate (ta) at (-.72,-.22);
  \coordinate (tb) at ( .00,-.60);
  \coordinate (tc) at ( .04, .80);
  \coordinate (td) at ( .72,-.18);
  \draw[tethidden] (ta)--(td);
  \draw[tetedge] (ta)--(tb)--(td);
  \draw[tetedge] (tc)--(ta) (tc)--(tb) (tc)--(td);
}
\newcommand{\BanschNet}{%
  \draw[faceedge] (-.55,0)--(.55,0)--(0,-.95)--cycle;
  \draw[faceedge] (-.55,0)--(0,.95)--(.55,0);
  \draw[faceedge] (-.55,0)--(-1.10,-.95)--(0,-.95);
  \draw[faceedge] (.55,0)--(1.10,-.95)--(0,-.95);
  \node[vertex,left] at (-.55,0) {$a$};
  \node[vertex,right] at (.55,0) {$b$};
  \node[vertex,below] at (0,-.95) {$c$};
  \node[vertex,above] at (0,.95) {$d$};
  \node[vertex,below left] at (-1.10,-.95) {$d$};
  \node[vertex,below right] at (1.10,-.95) {$d$};
}
\newcommand{\BanschAB}{%
  \draw[facemark] (-.47,.07)--(.47,.07);
  \draw[facemark] (-.47,-.07)--(.47,-.07);
}
\newcommand{\BanschAC}{\draw[facemark] (-.57,-.10)--(-.10,-.90);}
\newcommand{\BanschBC}{\draw[facemark] (.57,-.10)--(.10,-.90);}
\newcommand{\BanschBD}{\draw[facemark] (.53,-.10)--(1.00,-.90);}
\newcommand{\BanschCDleft}{%
  \draw[facemark] (-1.00,-.88)--(-.10,-.88);
}
\newcommand{\BanschCDright}{%
  \draw[facemark] (.10,-.88)--(1.00,-.88);
}

\begin{scope}
  \begin{scope}[yshift=1.45cm,scale=1.00]
    \BanschTetra
    \draw[dualmark] (ta)--(tb);
    \draw[tetmark] (ta)--(tc) (tb)--(td);
  \end{scope}
  \begin{scope}[yshift=-.60cm]
    \BanschNet\BanschAB\BanschAC\BanschBD
  \end{scope}
  \node[align=center] at (0,-2.12)
    {\emph{red}\\[-1pt]$([a,c],[b,d])$};
\end{scope}

\begin{scope}[xshift=2.9cm]
  \begin{scope}[yshift=1.45cm,scale=1.00]
    \BanschTetra
    \draw[dualmark] (ta)--(tb);
    \draw[tetmark] (ta)--(tc) (tb)--(tc);
  \end{scope}
  \begin{scope}[yshift=-.60cm]
    \BanschNet\BanschAB\BanschAC\BanschBC
  \end{scope}
  \node[align=center] at (0,-2.12)
    {\emph{black}\\[-1pt]$([a,c],[b,c])$};
\end{scope}

\begin{scope}[xshift=5.8cm]
  \begin{scope}[yshift=1.45cm,scale=1.00]
    \BanschTetra
    \draw[dualmark] (ta)--(tb);
    \draw[tetmark] (ta)--(tc) (tc)--(td);
  \end{scope}
  \begin{scope}[yshift=-.60cm]
    \BanschNet\BanschAB\BanschAC\BanschCDright
  \end{scope}
  \node[align=center] at (0,-2.12)
    {\emph{unnamed}\\[-1pt]$([a,c],[c,d])$};
\end{scope}

\begin{scope}[xshift=8.7cm]
  \begin{scope}[yshift=1.45cm,scale=1.00]
    \BanschTetra
    \draw[dualmark] (ta)--(tb);
    \draw[dualmark] (tc)--(td);
  \end{scope}
  \begin{scope}[yshift=-.60cm]
    \BanschNet\BanschAB\BanschCDleft\BanschCDright
  \end{scope}
  \node[align=center] at (0,-2.12)
    {\emph{unnamed}\\[-1pt]$([c,d],[c,d])$};
\end{scope}
\end{tikzpicture}
\caption{The four intrinsic marking configurations relative to a selected
global refinement edge $[a,b]$, shown as marked tetrahedra (top) and unfolded
face nets (bottom).  The first two configurations are called red and black
by B\"ansch; the last two are unnamed.  Thick strokes are face marks, with
double strokes when both incident faces mark the same geometric edge.}
\label{fig:bansch-configurations}
\end{fixedfigure}

B\"ansch notes that admissible markings can be constructed by making the
longest edge of every face unique through a hypothetical perturbation and
marking that edge~\cite[p.~185]{Bansch1991}.

\subsection{One B\"ansch bisection}

Let $K$ be an active tetrahedron and suppose that
$[a,b]\in\mathcal G(K)$ is the edge selected for its bisection.  Write the
remaining vertices as $c,d$ and put
\[
 z:=\operatorname{mid}([a,b]),
 \qquad
 K_a=[a,z,c,d],
 \qquad
 K_b=[b,z,c,d].
\]
Every initial tetrahedron carries the state $\varepsilon_K=0$.  If $K$ is
noninitial with immediate parent $T$, its state is
\[
 \varepsilon_K=1
 \quad\Longleftrightarrow\quad
 K\text{ and }T\text{ are both black}.
\]
The children and their data are determined by the following rules.
\begin{enumerate}[label=\textnormal{(R\arabic*)},leftmargin=2.8em]
\item An inherited face retains its mark from $K$.
\item A cut face is marked by its edge opposite $z$; thus
\[
\begin{aligned}
 \mu_{[a,z,c]}&=[a,c],&\qquad \mu_{[a,z,d]}&=[a,d],\\
 \mu_{[b,z,c]}&=[b,c],&\qquad \mu_{[b,z,d]}&=[b,d].
\end{aligned}
\]
\item If $\varepsilon_K=0$, the common new face is marked by
\[
 \mu_{[z,c,d]}=[c,d].
\]
If $\varepsilon_K=1$, then $K$ is black, so there is a unique
$w\in\{c,d\}$ with $(\mu_a,\mu_b)=([a,w],[b,w])$.  The exceptional mark is
\[
 \mu_{[z,c,d]}=[z,w].
\]
\item After the face marks have been assigned, set
\[
 \varepsilon_{K_i}=1
 \quad\Longleftrightarrow\quad
 K_i\text{ and }K\text{ are both black},
 \qquad i\in\{a,b\}.
\]
\end{enumerate}
These are B\"ansch's local three-dimensional rules, written in intrinsic
face notation.

\begin{fixedfigure}
\begin{tikzpicture}[
  every node/.style={font=\footnotesize},
  tetedge/.style={line width=.45pt,line join=round},
  hidden/.style={densely dashed,line width=.40pt},
  selected/.style={line width=1.15pt,line cap=round},
  newedge/.style={densely dashed,line width=.85pt},
  faceedge/.style={line width=.45pt},
  facemark/.style={line width=1.15pt,line cap=round},
  flow/.style={-{Latex[length=2mm]},line width=.45pt}
]
\begin{scope}[xshift=-3.40cm]
  \coordinate (a) at (-2.0,-.70);
  \coordinate (b) at ( 2.0,-.70);
  \coordinate (c) at ( .65,1.25);
  \coordinate (d) at (-.15,.08);
  \coordinate (z) at (0,-.70);
  \fill[black!5] (z)--(c)--(d)--cycle;
  \draw[tetedge] (a)--(c)--(b)--cycle;
  \draw[tetedge] (a)--(d)--(b);
  \draw[hidden] (c)--(d);
  \draw[selected] (a)--(b);
  \draw[newedge] (z)--(c) (z)--(d);
  \fill (a) circle (1.15pt) node[below left] {$a$};
  \fill (b) circle (1.15pt) node[below right] {$b$};
  \fill (c) circle (1.15pt) node[above] {$c$};
  \fill (d) circle (1.15pt) node[above left] {$d$};
  \fill (z) circle (1.50pt) node[below] {$z$};
  \node at (0,-1.38)
    {$K_a=[a,z,c,d]$,\qquad $K_b=[b,z,c,d]$};
\end{scope}

\draw[flow] (-1.75,.28)--(-.10,.28)
  node[midway,above,align=center,font=\scriptsize]
  {common new face\\$[z,c,d]$};

\begin{scope}[xshift=1.25cm]
  \coordinate (zo) at (0,-.30);
  \coordinate (co) at (-.88,1.00);
  \coordinate (do) at (.88,1.00);
  \draw[faceedge] (zo)--(co)--(do)--cycle;
  \draw[facemark] (co)--(do);
  \fill (zo) circle (1.05pt) node[below] {$z$};
  \fill (co) circle (1.05pt) node[above right] {$c$};
  \fill (do) circle (1.05pt) node[above right] {$d$};
  \node[align=center,text width=2.1cm,font=\scriptsize] at (0,-1.18)
    {ordinary case\\$\varepsilon_K=0$\\$\mu_{[z,c,d]}=[c,d]$};
\end{scope}

\begin{scope}[xshift=4.45cm]
  \coordinate (ze) at (0,-.30);
  \coordinate (ce) at (-.88,1.00);
  \coordinate (de) at (.88,1.00);
  \draw[faceedge] (ze)--(ce)--(de)--cycle;
  \draw[facemark] (ze)--(ce);
  \fill (ze) circle (1.05pt) node[below] {$z$};
  \fill (ce) circle (1.05pt) node[above left] {$c=w$};
  \fill (de) circle (1.05pt) node[above right] {$d$};
  \node[align=center,text width=2.6cm,font=\scriptsize] at (0,-1.18)
    {black--black case\\
     $\varepsilon_K=1$\\$\mu_{[z,c,d]}=[z,w]$};
\end{scope}
\end{tikzpicture}
\caption{One B\"ansch bisection and the two possible marks of its common new
face.  The ordinary case uses $[c,d]$, whereas the black--black case uses
$[z,w]$.  Here $w=c$; the case $w=d$ is symmetric.}
\label{fig:bansch-bisection}
\end{fixedfigure}

\begin{algorithm}[H]
\caption{One B\"ansch bisection \BanschBisectTet}
\label{alg:banschbisecttet}
\KwIn{An active tetrahedron $K$ and a selected edge
$[a,b]\in\mathcal G(K)$}
\KwOut{the two marked children $\{K_a,K_b\}=\BanschBisectTet(K,[a,b])$ with
their states}
Write $\mathcal{V}(K)\setminus\{a,b\}=\{c,d\}$\;
$z\gets\operatorname{mid}([a,b])$\;
$K_a\gets[a,z,c,d]$ and $K_b\gets[b,z,c,d]$\;
Assign the face marks and child states according to
\textnormal{(R1)}--\textnormal{(R4)}\;
\Return $\{K_a,K_b\}$\;
\end{algorithm}

\begin{proposition}[B\"ansch local structure]
\label{prop:bansch-local}
The two children of a B\"ansch bisection have the same configuration and
unique global refinement edges.  Thus only an unbisected tetrahedron of the
initial mesh can have two global refinement edges.  Writing
$\mathrm{black}_i$ for a black tetrahedron $K$ with $\varepsilon_K=i$, the
transition is
\[
 \mathrm{red}\longrightarrow\mathrm{black}_0
 \longrightarrow\mathrm{black}_1
 \longrightarrow\mathrm{red},
 \qquad
 \mathrm{unnamed}\longrightarrow\mathrm{black}_0.
\]
\end{proposition}

\begin{proof}
The configuration transition is the three-dimensional case analysis in
\cite{Bansch1991}, sharpened for the unnamed
configurations by the AMP type transition in
\cite[Section~2]{AMP2000}.

It remains to check uniqueness.  For $K_a$, the four face marks are
$\mu_a$, $[a,c]$, $[a,d]$, and the new-face mark.  If $\varepsilon_K=0$, the
last mark is $[c,d]$; according as $\mu_a=[a,c]$, $[a,d]$, or $[c,d]$, the
unique global refinement edge of $K_a$ is $[a,c]$, $[a,d]$, or $[c,d]$.  If
$\varepsilon_K=1$, then $K$ is black, so $\mu_a=[a,w]$ and the new-face mark
is $[z,w]$ for one $w\in\{c,d\}$; the unique global refinement edge is then
$[a,w]$.
The argument for $K_b$ is symmetric.  Thus every child has a unique global
refinement edge, which also proves the ambiguity assertion.
\end{proof}

\subsection{B\"ansch refinement}

An \emph{admissible initial B\"ansch marked mesh} $\mathcal B_0$ consists of a
finite conforming tetrahedral mesh, an admissible intrinsic face marking, and
the state $\varepsilon_K=0$ on every $K\in\mathcal B_0$.  A current
\emph{B\"ansch marked mesh} $\mathcal B$ is the active mesh, intrinsic face marks,
and states obtained from $\mathcal B_0$ by the local rule above.  In
particular, its states are consistent with \textnormal{(R4)}.  Moreover,
every active $K$ satisfies $\mathcal G(K)\ne\varnothing$: an unbisected
initial tetrahedron retains its nonempty set $\mathcal G(K)$, whereas every
descendant has a unique global refinement edge by
\Cref{prop:bansch-local}.
A B\"ansch marked mesh is called \emph{conforming} if its underlying active
tetrahedral mesh is conforming.

Only an initial tetrahedron $K$ with $\card\mathcal G(K)=2$ requires a
choice.  It may remain active through several refinement calls; the choice is
made when it is first bisected.  A \emph{selection strategy} $\mathcal A$
assigns one edge $\mathcal A(\mathcal H;K)\in\mathcal G(K)$ at that event,
possibly as a function of the preceding batch history $\mathcal H$.  For any
active tetrahedron selected for bisection, let
$e_{\mathcal A}(\mathcal H;K)$ denote this choice in the
initial two-edge case and the unique member of $\mathcal G(K)$ otherwise.
We suppress $\mathcal H$ below; all choices in one batch use the same batch-start
history.

\begin{algorithm}[H]
\caption{Simultaneous B\"ansch bisection \BanschBisectTets}
\label{alg:banschbisecttets}
\KwIn{A B\"ansch marked mesh $\mathcal B$, a set
$\mathcal S\subseteq\mathcal B$, and a selection strategy $\mathcal A$}
\KwOut{the B\"ansch marked mesh $\BanschBisectTets_{\mathcal A}(\mathcal B,\mathcal S)$
obtained by bisecting every $K\in\mathcal S$}

$\widehat{\mathcal B}\gets\mathcal B\setminus\mathcal S$\;

\For{$K\in\mathcal S$}{
  $e_K\gets e_{\mathcal A}(K)$\;
  $\mathcal C_K\gets\BanschBisectTet(K,e_K)$\;
  $\widehat{\mathcal B}\gets\widehat{\mathcal B}\cup\mathcal C_K$\;
}

\Return $\widehat{\mathcal B}$\;
\end{algorithm}

All edges used in a batch are determined from the fixed input marked mesh
$\mathcal B$; no child created in the batch is bisected again within that
batch.

\begin{algorithm}[H]
\caption{One B\"ansch refinement call \BanschRefine}
\label{alg:banschrefine}
\KwIn{A conforming B\"ansch marked mesh $\mathcal B$, a marking set
$\M\subseteq\mathcal B$, and a selection strategy $\mathcal A$}
\KwOut{upon termination, the B\"ansch marked mesh
$\BanschRefine_{\mathcal A}(\mathcal B,\M)$ without hanging vertices}
$\mathcal S\gets\M$\;
\While{$\mathcal S\ne\varnothing$}{
  $\mathcal B\gets
    \BanschBisectTets_{\mathcal A}(\mathcal B,\mathcal S)$\;
  $\mathcal S\gets
    \{K\in\mathcal B:K\text{ has a hanging vertex}\}$\;
}
\Return $\mathcal B$\;
\end{algorithm}

\Cref{alg:banschrefine} is B\"ansch's global three-dimensional refinement
loop, with the deferred initial two-edge choices made explicit through
$\mathcal A$.

For the three-dimensional algorithm, B\"ansch proves termination, conformity,
and stability.  That formulation
leaves a possible two-edge choice implicit.  Termination uniformly over the
selection strategies above will instead follow from the deferred-resolution
argument below.

\subsection{Finite resolution and AMP correspondence}

\begin{definition}[Resolution space]\label{def:bansch-resolution-space}
For an admissible initial B\"ansch marked mesh $\mathcal B_0$, let
\[
 \mathfrak D_0
 :=
 \{K\in\mathcal B_0:\card\mathcal G(K)=2\}.
\]
A \emph{resolution} is a simultaneous choice of one edge
$\rho(K)\in\mathcal G(K)$ for every $K\in\mathfrak D_0$.  The set of all
resolutions is
\begin{equation}\label{eq:bansch-resolution-space}
 \Sigma_{\mathrm B}
 :=
 \prod_{K\in\mathfrak D_0}\mathcal G(K).
\end{equation}
\end{definition}

\begin{lemma}[Finite initial resolution]\label{lem:bansch-resolution-space}
The resolution space is finite, with
\[
 \card\Sigma_{\mathrm B}=2^{\card\mathfrak D_0}.
\]
Moreover, a resolution fixes every possible initial ambiguity, while all
descendant refinement edges are unique.
\end{lemma}

\begin{proof}
By \Cref{lem:bansch-two-global}, every factor in
\eqref{eq:bansch-resolution-space} consists of two opposite edges, whereas
every initial tetrahedron in
$\mathcal B_0\setminus\mathfrak D_0$ has a unique global refinement edge.
Since the initial mesh is finite, the product is finite and has the stated
cardinality.  The choices in different factors are independent because they
do not alter the fixed intrinsic face marks.  The assertion about descendants
is \Cref{prop:bansch-local}.
\end{proof}

For $\rho\in\Sigma_{\mathrm B}$, select $\rho(K)$ as the refinement edge
of every $K\in\mathfrak D_0$, select the unique global refinement edge of every other
initial tetrahedron, retain the intrinsic face marks, and unset all AMP flags.
The resulting initial AMP marked mesh is denoted by $\mathcal{T}_0^\rho$.
Thus $\mathcal B_0$ and $\mathcal{T}_0^\rho$ have the same underlying
tetrahedral mesh and intrinsic face marks, while the AMP refinement edges
encode the resolution of every B\"ansch ambiguity.

With the selected global edge as the AMP refinement edge, red,
$\mathrm{black}_0$, $\mathrm{black}_1$, and the unnamed configurations with
one or two occurrences of $[c,d]$ are, respectively, the AMP types
$A,P_u,P_f,M,O$.  Hence the state evolution is
\[
 A\longrightarrow P_u\longrightarrow P_f\longrightarrow A,
 \qquad
 M,O\longrightarrow P_u.
\]

\begin{proposition}[One-step B\"ansch--AMP correspondence]
\label{prop:bansch-amp-onestep}
After a global refinement edge has been selected, one call of
$\BanschBisectTet$ and one application of $\BisectTet$ produce the same
geometric children, face marks, and refinement edges, with corresponding
B\"ansch states and AMP types.
\end{proposition}

\begin{proof}
This is the local equivalence established in
\cite[Section~2]{AMP2000} under the identification above.
\end{proof}

\begin{theorem}[Deferred resolution]\label{thm:bansch-deferred}
Let $\mathcal B_0$ be an admissible initial B\"ansch marked mesh.  For every finite
or infinite B\"ansch refinement history from $\mathcal B_0$ generated by
selection strategies, there exists a resolution
$\rho\in\Sigma_{\mathrm B}$, depending on the history but fixed throughout
it, such that, at every complete-batch boundary, the B\"ansch execution and
the AMP execution from $\mathcal{T}_0^\rho$ using the same marking sets have the same
active tetrahedra, face marks, and corresponding states.  Every bisection
performed up to that boundary uses the same geometric edge in both
executions.
\end{theorem}

\begin{proof}
For each $K\in\mathfrak D_0$ that is ever bisected, let $\rho(K)$ be the
edge selected at that event; choose either edge for an ambiguous initial
tetrahedron that is never bisected.  This defines one fixed resolution after
the history is known.

The initial geometries, face marks, and states agree under the identification
above.  By construction, the selected edge also agrees at the first
bisection of every ambiguous initial tetrahedron.  All descendants have a
unique global refinement edge by
\Cref{prop:bansch-local}; hence all later bisection edges agree automatically.
By \Cref{prop:bansch-amp-onestep}, the two histories have identical children,
marks, and states by induction.  Therefore the intermediate meshes have the
same hanging vertices, and the two conformity loops select the same
tetrahedra.  Induction over the batches proves the assertions at every
complete-batch boundary.
\end{proof}

\subsection{Uniform B\"ansch closure complexity}

A finite \emph{B\"ansch adaptive process} issued from an admissible
initial B\"ansch marked mesh $\mathcal B_0$ specifies an integer $L\ge0$ and, at each
completed step $\ell<L$, a marking set
$\M_\ell\subseteq\mathcal B_\ell$ and a selection strategy $\mathcal A_\ell$.
Both choices may depend on the preceding history; formally,
\[
 \mathcal B_{\ell+1}
 :=
 \BanschRefine_{\mathcal A_\ell}(\mathcal B_\ell,\M_\ell),
 \qquad \ell=0,\ldots,L-1.
\]

\begin{theorem}[Unconditional B\"ansch closure estimate]
\label{thm:bansch-closure}
Let $\mathcal B_0$ be an admissible initial B\"ansch marked mesh.  Every finite
B\"ansch adaptive process issued from $\mathcal B_0$ is well defined:
all calls terminate, and the resulting sequence satisfies
\begin{equation}\label{eq:bansch-closure}
 \card\mathcal B_L-\card\mathcal B_0
 \le C_{\mathrm B}(\mathcal B_0)
       \sum_{\ell=0}^{L-1}\card\M_\ell,
\end{equation}
where $C_{\mathrm B}(\mathcal B_0)<\infty$ depends only on $\mathcal B_0$.
\end{theorem}

\begin{proof}
Suppose that one conformity loop were infinite.  Append its batches to the
preceding finite history.  By \Cref{thm:bansch-deferred}, there exists a
resolution $\rho\in\Sigma_{\mathrm B}$, depending on this infinite history
but fixed throughout it, such that every prefix ending at a batch boundary
agrees with the AMP execution from $\mathcal{T}_0^\rho$.  The AMP completion loop
would therefore also have infinitely many batches,
contradicting \Cref{thm:AMPtermination}.  Hence every B\"ansch call terminates;
the same correspondence and \Cref{thm:AMPtermination} show that its output is
conforming.

For the resulting finite adaptive history, \Cref{thm:bansch-deferred} gives a
resolution $\rho\in\Sigma_{\mathrm B}$, depending on this history but fixed
throughout it, for which the B\"ansch and AMP histories have the same active
meshes at every batch boundary.  Applying \Cref{thm:unconditional} and then
taking the maximum over the finite set in
\Cref{lem:bansch-resolution-space} proves \eqref{eq:bansch-closure}.  Thus one
may take
\[
 C_{\mathrm B}(\mathcal B_0)
 :=
 \max_{\rho\in\Sigma_{\mathrm B}}
 C_{\mathrm{clos}}(\mathcal{T}_0^\rho)<\infty.
\]
The proof is complete.
\end{proof}

\begin{corollary}[Order-based admissible initialization]
\label{cor:bansch-order-initialization}
Every finite conforming tetrahedral mesh admits an initial face marking to
which \Cref{thm:bansch-closure} applies: impose a strict global order on its
geometric edges, mark each face by its maximal edge, and set
$\varepsilon_K=0$.
\end{corollary}

\begin{proof}
The maximal edge of a tetrahedron is maximal on both faces containing it and
therefore belongs to $\mathcal G(K)$.  The marking is intrinsic because the
order is global.
\end{proof}

\section{Numerical experiments}\label{sec:numerics}
\raggedbottom

This section compares AMP with the DGS initialization~\cite{DGS} and the
weakly compatible AGK initialization~\cite{AGK2018}.  The B\"ansch routine is
not included as a separate numerical competitor.  Once its ambiguous choices
are fixed according to a corresponding AMP initialization,
\Cref{prop:bansch-amp-onestep,thm:bansch-deferred} give the same refinement history event by event,
so separate convergence curves would be redundant.  These computations are
illustrative: no theoretical comparison with DGS or AGK is asserted.

\subsection{Experimental setting}

We use two three-dimensional test domains.  Their initial meshes were
generated with Gmsh~\cite{GeuzaineRemacle2009}.  The experiments were run in
MATLAB R2024b.  For each initial
mesh and refinement rule we solve
\begin{equation}\label{eq:numerical-model}
 -\Delta u=1\quad\text{in }\Omega,
 \qquad
 u=0\quad\text{on }\partial\Omega
\end{equation}
with continuous piecewise affine finite elements.  The adaptive loop uses the
squared residual error indicators and D\"orfler marking with parameter
$\theta=0.3$, and stops at the first iterate with at least $10^5$ degrees of
freedom.  Because an exact solution is unavailable on these domains, the
figures report the residual estimator $\eta_\ell$, rather than an exact energy
error.

For a tetrahedron $K$, let
\[
 \gamma(K):=\frac{\diam(K)}{2r_K},
 \qquad
 \gamma(\mathcal{T}_\ell):=\max_{K\in\mathcal{T}_\ell}\gamma(K),
\]
where $r_K$ is the inradius.  We report the normalized shape parameter
$\gamma(\mathcal{T}_L)/\gamma(\mathcal{T}_0)$ and the realized closure ratio
\begin{equation}\label{eq:empirical-closure}
 C_{\mathrm{emp}}
 :=
 \frac{\card\mathcal{T}_L-\card\mathcal{T}_0}
      {\displaystyle\sum_{\ell=0}^{L-1}\card\M_\ell}.
\end{equation}
The latter is one history-dependent observation and hence only a lower bound
for a worst-case closure constant; it should not be confused with the
explicit upper bounds proved above.

\begin{figure}[htbp]
\centering
\includegraphics[width=\textwidth]{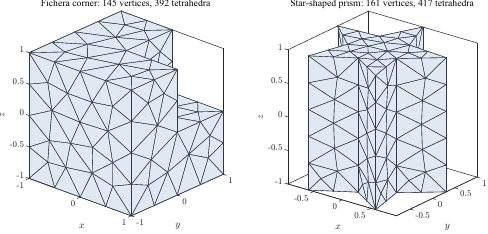}
\caption{The initial Gmsh-generated tetrahedral meshes.  The displayed
triangles are the boundary faces of the tetrahedralizations; the panel titles
give the numbers of vertices and tetrahedra.}
\label{fig:initial-meshes}
\end{figure}

\subsection{Fichera corner}

The domain is the cube $[-1,1]^3$ with the positive octant removed.  We use
the unstructured Gmsh mesh shown in the left panel of
\Cref{fig:initial-meshes}; its greedy DGS color parameter is
$N_{\mathrm{DGS}}=9$.  The
complete terminal statistics are given in \Cref{tab:fichera-results}.

\begin{table}[H]
\caption{Fichera-corner results.  Here $L$ is the number of adaptive steps and $N_L$ is the number of degrees of freedom.}
\label{tab:fichera-results}
\centering
\small
\begin{tabular}{@{}lrrrrrr@{}}
\toprule
 Method & $L$ & $N_L$ & $\card\mathcal{T}_L$
& $\gamma(\mathcal{T}_L)/\gamma(\mathcal{T}_0)$ & $C_{\mathrm{emp}}$
& $10\eta_L$ \\
\midrule
 AMP      & 21 & 101206 & 524126 & 2.060 & 2.786 & $1.3017$ \\
      DGS      & 22 & 119158 & 610731 & 3.872 & 2.912 & $1.5883$ \\
      AGK      & 21 & 110486 & 565448 & 3.860 & 2.895 & $1.4699$ \\
\bottomrule
\end{tabular}
\end{table}

\begin{figure}[H]
\centering
\includegraphics[width=.95\textwidth]{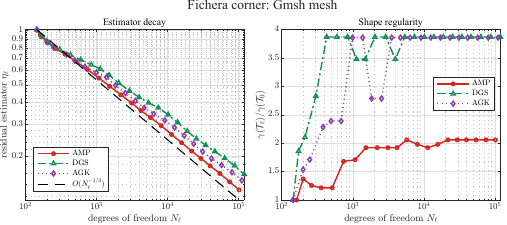}
\caption{Fichera corner.  Left: residual estimator versus degrees of freedom;
the dashed black line has reference slope $-1/3$.  Right: normalized
shape parameter along the adaptive histories.}
\label{fig:fichera-results}
\end{figure}

\subsection{Star-shaped prism}

The second domain is the extrusion over $-1\le z\le1$ of a ten-vertex star
whose alternating outer and inner radii are $1$ and $0.45$.  Gmsh's Delaunay
three-dimensional algorithm provides the mesh shown in
\Cref{fig:initial-meshes}; its greedy DGS color parameter is
$N_{\mathrm{DGS}}=7$.

\begin{table}[H]
\caption{Star-prism numerical results.}
\label{tab:starprism-results}
\centering
\small
\begin{tabular}{@{}lrrrrrr@{}}
\toprule
 Method & $L$ & $N_L$ & $\card\mathcal{T}_L$
& $\gamma(\mathcal{T}_L)/\gamma(\mathcal{T}_0)$ & $C_{\mathrm{emp}}$
& $10^2\eta_L$ \\
\midrule
 AMP      & 21 & 108876 & 572380 & 2.058 & 2.815 & $6.1657$ \\
          DGS      & 22 & 104747 & 543771 & 4.554 & 3.053 & $8.5609$ \\
          AGK      & 22 & 108077 & 557750 & 4.554 & 2.985 & $7.5111$ \\
\bottomrule
\end{tabular}
\end{table}

\begin{figure}[H]
\centering
\includegraphics[width=.95\textwidth]{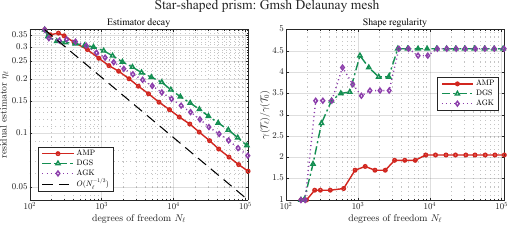}
\caption{Star-shaped prism.  Left: residual estimator versus degrees of
freedom; the dashed black line has reference slope $-1/3$.  Right:
normalized shape parameter along the adaptive histories.}
\label{fig:starprism-results}
\end{figure}

\subsection{Sensitivity to the initial vertex numbering}

To separate geometric effects from dependence on the input labels, we keep
each physical initial mesh fixed and repeat the experiment after five
independent random permutations of its vertex numbers. The
coordinates and physical tetrahedra are unchanged.  Trial zero uses the
original numbering.  All six trials use the same model problem, marking
parameter, and stopping criterion as above.

For this subsection, $\refedge(K)$ denotes the initial refinement edge assigned
to $K$ by the method and initialization under consideration.  Define its mean
normalized length by
\begin{equation}\label{eq:mean-refinement-edge-ratio}
 \overline r_0
 :=\frac{1}{\card\mathcal{T}_0}\sum_{K\in\mathcal{T}_0}
 \frac{|\refedge(K)|}{\diam(K)}.
\end{equation}
Every entry in \Cref{tab:numbering-sensitivity} is the mean over the six
numberings followed by the range in brackets.

\begin{table}[htbp]
\caption{Sensitivity to the initial vertex numbering.  Each entry is the
six-numbering mean followed by $[\min,\max]$.}
\label{tab:numbering-sensitivity}
\centering
\small
\setlength{\tabcolsep}{2pt}
\begin{tabular}{@{}lcccc@{}}
\toprule
\multicolumn{5}{c}{Fichera corner} \\
\cmidrule(lr){1-5}
Method & $\overline r_0$ & $C_{\mathrm{emp}}$
& Shape ratio & $10\eta_L$ \\
\midrule
AMP & $\begin{array}{c}1.0000\\ {[1.0000,1.0000]}\end{array}$
 & $\begin{array}{c}2.7704\\ {[2.7386,2.7864]}\end{array}$
 & $\begin{array}{c}2.0596\\ {[2.0596,2.0596]}\end{array}$
 & $\begin{array}{c}1.2616\\ {[1.1792,1.3035]}\end{array}$ \\
DGS & $\begin{array}{c}0.8156\\ {[0.7987,0.8347]}\end{array}$
 & $\begin{array}{c}2.9423\\ {[2.9112,2.9774]}\end{array}$
 & $\begin{array}{c}3.6808\\ {[3.0132,4.4001]}\end{array}$
 & $\begin{array}{c}1.6339\\ {[1.5883,1.6541]}\end{array}$ \\
AGK & $\begin{array}{c}0.8091\\ {[0.8005,0.8206]}\end{array}$
 & $\begin{array}{c}3.0126\\ {[2.8950,3.0612]}\end{array}$
 & $\begin{array}{c}3.9019\\ {[2.4648,5.2520]}\end{array}$
 & $\begin{array}{c}1.6609\\ {[1.4699,1.7816]}\end{array}$ \\
\addlinespace
\multicolumn{5}{c}{Star-shaped prism} \\
\cmidrule(lr){1-5}
Method & $\overline r_0$ & $C_{\mathrm{emp}}$
& Shape ratio & $10^2\eta_L$ \\
\midrule
AMP & $\begin{array}{c}1.0000\\ {[1.0000,1.0000]}\end{array}$
 & $\begin{array}{c}2.8145\\ {[2.8145,2.8145]}\end{array}$
 & $\begin{array}{c}2.0577\\ {[2.0577,2.0577]}\end{array}$
 & $\begin{array}{c}6.1657\\ {[6.1657,6.1657]}\end{array}$ \\
DGS & $\begin{array}{c}0.8023\\ {[0.7880,0.8209]}\end{array}$
 & $\begin{array}{c}3.0556\\ {[2.9963,3.1217]}\end{array}$
 & $\begin{array}{c}4.3290\\ {[4.2436,4.5536]}\end{array}$
 & $\begin{array}{c}8.3929\\ {[7.8744,8.8280]}\end{array}$ \\
AGK & $\begin{array}{c}0.7959\\ {[0.7694,0.8210]}\end{array}$
 & $\begin{array}{c}3.0771\\ {[2.9851,3.1395]}\end{array}$
 & $\begin{array}{c}4.3265\\ {[4.2436,4.5536]}\end{array}$
 & $\begin{array}{c}8.1806\\ {[7.5111,8.6608]}\end{array}$ \\
\bottomrule
\end{tabular}
\end{table}

The test isolates the principal effect of the geometry-aware AMP
initialization.  For every initial tetrahedron, not merely on average, AMP
selects a diameter-carrying edge; hence the ratio in
\eqref{eq:mean-refinement-edge-ratio} is identically one under every
permutation.  DGS and AGK need not select a longest edge, and their mean ratios
vary with the numbering.  AMP also has an invariant final shape ratio on both
meshes and completely invariant terminal data on the star prism.  On the
Fichera mesh, equal-length tie breaking changes some AMP refinement histories,
but its observed closure and shape statistics remain markedly more stable
than those of DGS and AGK\@.

\subsection{Observed behavior}

Across the two initial meshes, the empirical closure ratios lie between
$2.786$ and $3.053$.  Direct implementation checks of B\"ansch against AMP
produced exactly the same recorded history on the star-prism mesh and only the
small difference caused by a distinct resolution of the initial ambiguity on
the Fichera mesh.

For these initializations, the final AMP shape ratios lie between $2.058$ and
$2.060$, whereas the DGS and AGK ratios lie between $3.860$ and $4.554$.  All
three displayed methods reduce the estimator steadily, but at comparable
terminal degrees of freedom AMP gives the smallest estimator in these tests.

The numbering-sensitivity test supports a geometry-based explanation for this
behavior.  AMP's strict global edge order uses Euclidean length as the primary
key, with vertex indices only breaking ties, whereas the DGS greedy coloring
and the AGK vertex ordering used here are combinatorial.  Cutting a
diameter-carrying edge can reduce the local scale more directly and may limit
unfavorable shape excursions.  Moreover, the AMP values
$C_{\mathrm{emp}}=2.786$ and $2.815$ for the original numberings are slightly
smaller than the corresponding DGS and AGK values $2.895$--$3.053$,
suggesting that less conforming completion was required in these particular
histories.

\providecommand{\bysame}{\leavevmode\hbox to3em{\hrulefill}\thinspace}
\providecommand{\MR}{\relax\ifhmode\unskip\space\fi MR }
\providecommand{\MRhref}[2]{%
  \href{http://www.ams.org/mathscinet-getitem?mr=#1}{#2}
}
\providecommand{\href}[2]{#2}

\end{document}